\documentclass[12pt]{article}
 \usepackage[margin=1in]{geometry} 
\usepackage{amsmath,amsthm,amssymb,amsfonts}
 \usepackage{amssymb}

\usepackage[utf8]{inputenc}
\usepackage[pagewise]{lineno}

\usepackage{hyperref}
 \usepackage{csquotes}
\usepackage[utf8]{inputenc}
\usepackage[english]{babel}
\usepackage{xcolor}
\usepackage{biblatex}
\title{Global existence of solutions to the chemotaxis system with logistic source under nonlinear Neumann boundary conditions}
\author{Minh Le }
\date{\today}

\begin{document}
\maketitle
\begin{abstract}
    We consider classical solutions to the chemotaxis system with logistic source $f(u) := au-\mu u^2$ under nonlinear Neumann boundary conditions $\frac{\partial u}{ \partial \nu } = |u|^{p}$ with $p>1$ in a smooth convex bounded domain $\Omega \subset \mathbb{R}^n$ where $n \geq 2$. This paper aims to show that if $p<\frac{3}{2}$, and $\mu >0$, $n=2$, or $\mu$ is sufficiently large when $n\geq 3$, then the parabolic-elliptic chemotaxis system admits a unique positive global-in-time classical solution that is bounded in $\Omega \times (0, \infty)$. The similar result is also true if $p<\frac{3}{2}$, $n=2$, and $\mu>0$ or $p<\frac{7}{5}$, $n=3$, and $\mu $ is sufficiently large for the parabolic-parabolic chemotaxis system. 
\end{abstract}

\numberwithin{equation}{section}
\newtheorem{theorem}{Theorem}[section]
\newtheorem{lemma}[theorem]{Lemma}
\newtheorem{remark}{Remark}[section]
\newtheorem{Prop}{Proposition}[section]
\newtheorem{Def}{Definition}[section]
\newtheorem{Corollary}{Corollary}[theorem]
\allowdisplaybreaks

\section{Introduction} \label{Introduction}
We are concerned in this paper with solutions to the chemotaxis model as follows:
\begin{equation} \label{1.1}
    \begin{cases} \tag{KS}
u_{t}=  \Delta u - \chi \nabla \cdot (u \nabla v) +au -\mu u^2   \qquad &x\in {\Omega},\, t \in (0,T_{\rm max }), \\ 
 \tau v_t  = \Delta v+\alpha u - \beta v \qquad &x\in {\Omega},\, t \in (0,T_{\rm max }),
\end{cases} 
\end{equation}
in a smooth, convex, bounded domain $\Omega \subset \mathbb{R}^n$ where $\alpha,\beta, a, \mu >0$, $\tau \geq 0$, and $\chi \in \mathbb{R}$. The system \eqref{1.1} is complemented with the nonnegative initial conditions in $C^{2+\gamma}(\Omega)$, where $\gamma \in (0,1)$, not identically zero:
\begin{equation} \label{1.1.2}
    u(x,0) = u_0(x), \qquad v(x,0)=v_0(x), \qquad x\in \Omega,
\end{equation}
and the nonlinear Neumann boundary conditions
\begin{equation} \label{boundary-data}
    \frac{\partial u}{\partial \nu } = |u|^p, \qquad \frac{\partial v}{\partial \nu }  = 0, \qquad x\in \partial \Omega,\, t \in (0, T_{\rm max}).
\end{equation}
where $p>1$ and $\nu $ is the outward normal vector. The PDE system in \eqref{1.1} is used in mathematical biology to model the mechanism of chemotaxis, that is, the movement of an organism in response to a chemical stimulus.  Here $u(x,t),v(x,t)$ represent the cell density and the concentration of chemical substances at position $x$ at the time $t$, respectively (\cite{Keller}). Historically, the system \eqref{1.1} without the logistic source under homogeneous Neumann boundary condition well-known as the most simplified Keller-Segel system has been intensively studied in many various directions such as global boundedness solutions, blow-up solutions, steady states, and blow-up sets.  A significant milestone in the derivation of the model, along with many interesting results, can be found in \cite{Horstmann}. Additionally, \cite{Winkler-2015} provides a comprehensive summary of the developed techniques and results accumulated over decades to address chemotaxis systems. For a more detailed exploration of the derivation of the system \eqref{1.1}, consider referring to \cite{Hillen}, which offers a broad survey of variations of chemotaxis models and their corresponding biological backgrounds. Furthermore, in two spatial dimension, the simplest form has been the focus of numerous research papers due to its intriguing mathematical property, the critical mass phenomena. This property states that if the initial mass is less than a certain number, solutions exist globally and remain bounded in time, while if the mass is larger than that value, solutions blow up in finite time. There are many research papers focusing on finding the precise value of critical mass. In two spatial dimensional domain, global existence and boundedness of solutions are studied in \cite{Dolbeault}, and \cite{Dolbeault1} for the sub-critical mass, and finite time blow-up solutions for the super-critical mass is investigated in \cite{Nagai1, Nagai2, Nagai4, Nagai3}. However, in higher spatial dimensional domain, this critical mass property is no longer true. In fact, it is proved in \cite{Winkler-2010} that finite time blow-up solutions can be constructed in a smooth bounded domain regardless of how small the mass is. \\
\indent We now shift our attention to the homogeneous Neumann boundary condition. In addition to the biological chemotaxis phenomena, the logistic term, $au-\mu u^2$, introduced in the evolution equation for $u$ plays a role in describing the growth of the population. Specifically, the term $au$, with $a \in \mathbb{R}$, is the growth rate of population and the term $-\mu u^2$ models additional overcrowding effects. It was investigated in \cite{TW} that the quadratic degradation term $-\mu u^2$ can prevent blow-up solutions. In fact, it was proven that if $\mu >\frac{n-2}{n}\chi \alpha$, and $\tau =0$, then the solutions exist globally and remain bounded at all time in a convex bounded domain with smooth boundary $\Omega \subset \mathbb{R}^n$, where $n \geq 2$. This result was later improved in \cite{Hu+Tao,KA,Tian5} that $\mu = \frac{n-2}{n}\chi \alpha$ can prevent blow-up when $\tau =0$. In a two-dimensional space with $\tau =1$, the system \eqref{1.1}  possesses a unique classical solution which is nonnegative and bounded in $\Omega \times (0, \infty)$ ( see e.g. \cite{ OY,OTYM}). These results were later improved in \cite{Tian4, Tian2} by replacing the logistic sources by sub-logistic ones such as $au-\mu \frac{u^2}{\ln^p(u+e)}$ for $p\in (0,1)$. In a higher-dimensional space with $\tau =1$, the similar results can be found in \cite{Winkler-2013} for the parabolic-parabolic model with an additional largeness assumption of $\mu$. In addition to the global classical solutions, the existence global weak solutions results for any arbitrary $\mu>0$ were also obtained in  \cite{TW} for the parabolic-elliptic models and  in \cite{Lankeit1} for the parabolic-parabolic models in a three-dimensional system. Furthermore, interested readers are referred to  \cite{Jian+Xiang, Lankeit2,LMLW, Li+Wang, LMZ, Tao+Winkler, MW2011, YCJZ, ZLBZ} to study more about qualitative and quantitative works of chemotaxis systems with logistic sources. \\
\indent The problem becomes more interesting and challenging if the homogeneous Neumann boundary condition is replaced by the nonlinear Neumann boundary condition. The method in this paper to obtain global boundedness results is first to establish a $L^1$ estimate, then for $L^{p_0}$ for some $p_0>1$, and finally apply a Moser-type iteration to obtain for $L^\infty$. Although this approach has been widely applied in treating global boundedness problems for reaction-diffusion equations (\cite{ Alikakos2,Alikakos1, Filo1989}), or for chemotaxis systems (\cite{Winkler-2013}), the main difficulties rely heavily on tedious integral estimations. Unlike the homogeneous Neumann boundary conditions, it is not even straightforward to see whether the total mass of the cell density function is globally bounded or not due to the nonlinear boundary term. In fact, most of the technical challenges in this paper are to deal with the nonlinear boundary term. Fortunately, the Sobolev's trace inequality enables us to solve a part of a problem:\\
\indent \textbf{Main Question:} \textit{"What is the largest value $p$ so that logistic damping still avoids blow-up?"} \\
This types of question for nonlinear parabolic equations has been intensively studied in 1990s. To be more precise, if we consider $\chi =0 $, our problem is similar to the following PDE:
\begin{equation} \tag{NBC}  \label{NBC}
    \begin{cases} 
U_{t}=  \Delta U-\mu U^Q  \qquad &x\in {\Omega},\, t \in (0,T_{\rm max }), \\ 
 \frac{\partial U}{\partial \nu } = U^P \qquad &x\in {\Omega},\, t \in (0,T_{\rm max }),  \\
 U(x,0) = U_0(x)  &x\in {\bar{\Omega}}.
\end{cases} 
\end{equation}
where  $\Omega $ is a smooth bounded domain in $\mathbb{R}^n$, $Q,P>1$, $\mu>0$ and $U_0 \in W^{1,\infty}(\Omega)$ is a nonnegative function. The study concerning the global existence was first investigated in \cite{CFQ}, and then improved in \cite{Quittner1} for $n \geq 2$. Particularly, it was shown that $P=\frac{Q+1}{2}$ is critical for the blow up in the following sense:
\begin{enumerate}
    \item if $P<\frac{Q+1}{2}$ then all solutions of \eqref{NBC} exist globally and are globally bounded,
    \item If $P>\frac{Q+1}{2}$ (or $P=\frac{Q+1}{2}$ and $\mu $ is sufficiently small ) then there exist initial functions $U_0$ such that the corresponding solutions of \eqref{NBC} blow-up in $L^{\infty}-$norm.
\end{enumerate}
\indent In comparison to our problem, we have $Q=2$ and $P=\frac{3}{2}$ is the critical power. Indeed, we also obtain the similar critical power $p=\frac{3}{2}$ as in Theorem \ref{p-e-e}. Notice that the local existence of positive solution was not mentioned in \cite{CFQ, Quittner1, Quittner}, and it is not clear for us to define $U^P$ without knowing $U$ is nonnegative, so the presence of absolute sign in \eqref{boundary-data} is necessary to obtain local positive solutions from nonnegative, not identically zero initial data.   \\
\indent Heuristically, the analysis diagram can be presented as follows. In case $\tau =0$, by substituting $-\Delta v= \alpha u-\beta v$ into the first equation of \eqref{1.1}, we obtain
 \[u_t= \Delta u +au -\chi \nabla u \cdot \nabla v+(\chi -\mu)u^2-\chi uv.\] If $\mu $ is sufficiently large, then solutions might be bounded globally since the nonlinear term $(\chi -\mu)u^2$ might dominate other terms including the nonlinear boundary term. In case $\tau =1$, we cannot substitute $\Delta v =\beta v-\alpha u$ directly into the first equation of \eqref{1.1}; however, we still have some certain controls of $ v$ by $u$ from the second equation of \eqref{1.1} thanks to Sobolev inequality. We expect that this intuition should be true in lower spacial dimension and "weaker" nonlinear boundary terms since the critical Sobolev exponent decreases if the spacial dimension increases. Indeed, our analysis does not work for $n \geq 4$ since we do not have enough rooms to control other positive nonlinear terms by using the term $(\chi -\mu )u^2$. One can also find similar ideas on sub-logistic source preventing 2D blow-up in \cite{Tian4}. \\
We summarize the main results to answer a part of the main question. Let us begin with the following theorem for the parabolic-elliptic case. 
\begin{theorem} \label{p-e-e}
Let  $\Omega $ be a bounded, convex domain with smooth boundary in $\mathbb{R}^n$ where $n \geq 2$, and $\tau =0$. If $\mu > \frac{n-2}{n}\chi \alpha$, and $1<p<\frac{3}{2}$ or $\mu= \frac{n-2}{n} \chi \alpha$ with $n\geq 3$ and $1<p<1+\frac{1}{n}$ then the system \eqref{1.1} with initial conditions \eqref{1.1.2} and boundary condition \eqref{boundary-data} possesses a unique positive classical solution which remains bounded in $\Omega \times (0, \infty)$.  
\end{theorem}
\begin{remark}
It is an open question whether there exists a classical finite time blow-up solution if $p \geq \frac{3}{2}$.
\end{remark}
\begin{remark}
   The proof of borderline boundedness in Theorem \ref{p-e-e} when $\mu = \frac{n-2}{n} \chi \alpha$ is adopted and modified from the arguments in \cite{Hu+Tao,KA,Tian5}.  However, applying Lemma \ref{Boundary-est} to overcome challenges in boundary integral estimations was not possible. Instead, we had to derive an alternative and improved estimation to handle the boundary term, which necessitated the condition of $p<1+\frac{1}{n}$.
\end{remark}
The next theorem  is for the parabolic-parabolic system in a two-dimensional space:
\begin{theorem} \label{2dthm}
Let  $\Omega$ be a bounded, convex domain with smooth boundary, and $\tau =1$, $n=2$, $1<p< \frac{3}{2}$, then the system \eqref{1.1} with initial conditions \eqref{1.1.2} and boundary condition \eqref{boundary-data} possesses a unique positive classical solution which remains bounded in $\Omega \times (0, \infty)$. 
\end{theorem}

In three-dimensional space, we prove the following theorem for the parabolic-parabolic case.
\begin{theorem} \label{3d}
Let  $\Omega$ be a bounded, convex domain with smooth boundary, and $\tau =1$, $n=3$, $1<p< \frac{7}{5}$, then there exists $\mu_0>0$ such that for every $\mu > \mu_0$, the system \eqref{1.1} with initial conditions \eqref{1.1.2} and boundary condition \eqref{boundary-data} possesses a unique positive classical solution which remains bounded in $\Omega \times (0, \infty)$. 
\end{theorem}
\begin{remark}
 Here we expect $p=\frac{7}{5}$ may not be the threshold of global boundedness and blow-up solutions, but rather the limitation of our analysis tools.
\end{remark}
\begin{remark}
 We leave the open question whether for every $n\geq 4$, there exists $p_n >1$ such that if $1<p<p_n$ solutions remain bounded in $\Omega \times (0, \infty)$.
\end{remark}
\begin{remark}
    In case $p=1$, one may adopt and modify the proof of Theorem \ref{p-e-e}, \ref{2dthm}, and \ref{3d} to obtain similar results.
\end{remark}
\indent The paper is organized as follows. The local well-possedness of solutions toward the system \eqref{1.1} including the short-time existence, positivity and uniqueness are established in Section \ref{local-wellposed}. In Section \ref{preliminaries}, we recall some basic inequalities and provide some essential estimates on the boundary of the solutions, which will be needed in the sequel sections. Section \ref{a priori estimates} is devoted to establishing the $L \log L$, and $L^r$ bounds for solutions. In Section \ref{globalbddn}, a Moser-type iteration scheme is applied to obtain an $L^\infty$ bound for solutions from an $L^r$ bound with a sufficiently large $r$. Finally, the main theorems are proved in Section \ref{proof of main theorems}.

\section{Local well-posedness} \label{local-wellposed}
In this section, we prove the short-time existence, uniqueness and positivity of solutions to the system \eqref{1.1}  under certain conditions of initial data. Although the proof just follows a basic fixed point argument, however we cannot find any suitable reference for our system. For the sake of completeness, here we will provide a proof, which is a modification of the proof of Theorem 1.1 in \cite{JVN}. First, we introduce some notations used throughout this section. We follow some definitions of Holder continuous spaces given in \cite{Lieberman}. 
For $k \in \mathbb{N}$, and $\gamma \in (0,1]$, we define the following norms and seminorms:
\begin{align*}
    [f]_{k+\gamma, \Omega_T } &:= \sum_{|\beta |+2j=k} \sup_{(x,t)\neq (y,s) \in \Omega_T} \frac{|D^{\beta}_x D^{j}_t\left ( f(x,t)-f(y,s) \right )|}{|x-y|^{\gamma}+|t-s|^{\frac{\gamma}{2}}}, \\
    \left \langle f \right \rangle _{k+\gamma, \Omega_T } &:= \sum_{|\beta|+2j=k-1} \sup_{(x,t)\neq (x,s) \in \Omega_T} \frac{|D^{\beta}_x D^{j}_t\left ( f(x,t)-f(x,s) \right )|}{|t-s|^{\frac{1+\gamma}{2}}},\\
    |f|_{k+\gamma, \Omega_T}&:= [f]_{k+\gamma, \Omega_T }+\left \langle f \right \rangle _{k+\gamma, \Omega_T },\\
    \left \| f \right \|_{C_{k+\gamma}( \Omega_T)} &:=  \sup_{(x,t)\in \Omega_T} |f(x,t)|+|f|_{k+\gamma, \Omega_T},
\end{align*}
where $\Omega_T := \Omega \times (0,T)$. We write $f \in C^{k+\gamma}(\Omega_T)$ if  $ \left \| f \right \|_{C_{k+\gamma}( \Omega_T)}< \infty$.
Now let us recall a useful result for the linear model. We consider the linear second order elliptic equation of non-divergence form:
\begin{equation} \label{local_eqn}
    Lu := u_t- a^{ij}D_{ij}u+ b^iD_iu+cu=f \text{ in } \Omega_T.
\end{equation}
Assume that there exists $\Lambda \geq \lambda >0$ such that 
\begin{equation} \label{local_1}
    \lambda |\xi|^2 \leq a^{ij}(x,t)\xi_i\xi_j \leq \Lambda |\xi|^2, \qquad (x,t) \in \Omega_T, \xi \in \mathbb{R}^n,
\end{equation}
where $a^{ij}, b^i, c \in C^{\gamma} (\bar{\Omega_T}) (0 <\gamma<1)$ and
\begin{equation} \label{local_2}
   \frac{1}{\lambda} \left \{ \sum _{i,j} \left \| a^{ij}  \right \|_{C^{\gamma}(\bar{\Omega}_T)}+\sum _{i}\left \| b^i \right \|_{C^{\gamma}(\bar{\Omega}_T)}+\left \| c \right \|_{C^{\gamma}(\bar{\Omega}_T)} \right \} \leq \Lambda_\gamma.
\end{equation}
\begin{theorem}[\cite{Lieberman}, p. 79, Theorem 4.31] \label{linear-local-existence-thm}
Let the assumptions \eqref{local_1}, \eqref{local_2} be in force, and  $\partial \Omega \in C^{2+\gamma} (0<\gamma<1).$ Let $f \in C^{\gamma}(\bar{\Omega_T})$, $g \in C^{1+\gamma}(\bar{\Omega}_T)$ and $u_0 \in C^{2+\gamma}(\bar{\Omega})$ satisfying the first order compatibility condition:
\begin{equation} \label{compare.1}
    \frac{\partial u_0}{\partial \nu  } =g(x,0) \text{ on }\partial \Omega.
\end{equation}
Then there exists a unique solution $u \in  C^{2+\gamma}(\bar{\Omega}_T)$ to the problem \eqref{local_eqn} with the Neumann boundary condition $\frac{\partial u}{\partial \nu } = g$ on $\partial \Omega \times (0,T)$. Moreover, there exists a constant $C$ independent of $g$ and $u_0$ such that
\begin{equation} \label{linear-est-local}
    \left \| u \right \|_{C^{2+\gamma}(\bar{\Omega}_T)} \leq C\left ( \frac{1}{\lambda}\left \| f \right \|_{C^{\gamma}(\bar{\Omega}_T)}+ \left \| g \right \|_{{C^{1+\gamma}(\bar{\Omega}_T)}} +\left \| u_0 \right \|_{C^{2+\gamma}(\bar{\Omega})} \right ).
\end{equation}
where $C$ is dependent only on $n, \gamma, \Lambda/ \lambda, \Lambda_\gamma$ and $\Omega$.
\end{theorem}
This estimate, together with Leray-Schauder fixed point argument is the main tools to prove the following theorem.  
\begin{theorem} \label{local-existence-theorem}
If nonnegative functions $u_0,v_0$  are in $C^{2+\gamma}(\bar{\Omega})$ such that \begin{equation}
    \frac{\partial u_0}{\partial \nu} = |u_0|^{1+\gamma} \qquad \text{ on } \partial \Omega,
\end{equation}
where $\gamma \in (0,1)$. Then there exists $T>0$ such that problem \eqref{1.1} admits 
a unique nonnegative solution $u,v$ in $C^{2+\gamma}(\bar{\Omega}_T)$. Moreover, if $u_0, v_0$ are not identically zero in $\Omega$ then $u,v$ are strictly positive in $\bar{\Omega_T}$. 
\end{theorem}
\begin{remark}
The convexity assumption of domain $\Omega$ is not necessary in this theorem.
\end{remark}
\begin{remark}
    By substituting $\gamma =p-1$ into Theorem \ref{local-existence-theorem}, we obtain local existence and uniqueness of positive solutions in Theorem \ref{p-e-e}, \ref{2dthm}, and \ref{3d}.  
\end{remark}
\begin{proof}
From now to the end of this proof, we will use $C$  as a universal notation for constants different from time to time. Firstly, the short-time existence of classical solution will be proved by a fixed point argument. Let $u \in C^{1+\gamma}(\bar{\Omega}_T)$ be such that $u(x,0)=u_0(x)$ in $\Omega$. Then, the functions $u_0$ and $g(x,t)= |u(x,t)|^{1+\gamma}$ satisfy condition \eqref{compare.1}, and $g \in  C^{1+\gamma}(\bar{\Omega}_T)$. We assume $T<1$, and consider the set of functions given by
\[
B_T(R):= \left \{ u \in C^{1+\gamma}(\bar{\Omega}_T) \text{ such that }   \left \| u \right \|_{ C^{1+\gamma}(\bar{\Omega}_T)} \leq R \right \}.
\]
Now we define the map 
\[
A: B_T(R) \longrightarrow  C^{1+\gamma}(\bar{\Omega}_T)
\]
where $Au:= U$ is a solution of 
\begin{equation} \label{local-existence-equation}
    \begin{cases} 
U_{t}=  \Delta U - \chi \nabla \cdot (u \nabla V) +au-\mu u^2  \qquad &x\in {\Omega},\, t \in (0,T_{\rm max }), \\ 
 \tau V_t  = \Delta V+ \alpha u - \beta V \qquad &x\in {\Omega},\, t \in (0,T_{\rm max }),
\end{cases} 
\end{equation}
under Neumann boundary condition: 
\begin{equation}
    \frac{\partial U}{\partial \nu } = |u|^{1+\gamma}, \qquad \frac{\partial V}{\partial \nu }=0, \qquad x\in \partial \Omega,\, t \in (0, T_{\rm max}),
\end{equation}
and initial data $\left ( U(x,0),V(x,0) \right )= \left (u_0(x),v_0(x) \right ) $ in $\Omega$. We first prove that $A$ sends bounded sets into relative compact sets of $C^{1+\gamma}(\bar{\Omega}_T)$. Indeed, the inequality \eqref{linear-est-local} implies there exists $R'>0$ independent of $T$ such that $\left \| Au \right \|_{C^{2+\gamma}(\bar{\Omega}_T)} \leq R'$ for all $u$ in $B_T(R)$. As bounded sets in $C^{2+\gamma}(\bar{\Omega}_T)$ are relatively compact in $C^{1+\gamma}(\bar{\Omega}_T)$. We claim that $A$ is continuous. In fact, let $u_n \to u$ in $C^{1+\gamma}(\bar{\Omega}_T)$, we need to prove $U_n:=Au_n \to  U:=Au$ in $C^{1+\gamma}(\bar{\Omega}_T)$. Now we can see that $U_n-U$ satisfies
\begin{equation} \label{local-existence-equation-1}
    \begin{cases} 
(U_n-U)_{t}=  \Delta (U_n-U)+ f_n,\, \qquad &x\in {\Omega}, \, t \in (0,T_{\rm max }), \\ 
 \tau (V_n-V)_t  = \Delta (V_n-V)+\alpha (u_n-u) - \beta (V_n-V) \qquad &x\in {\Omega},\, t \in (0,T_{\rm max }),
\end{cases} 
\end{equation}
where $f_n :=- \chi \nabla \cdot (u_n \nabla V_n- u \nabla V) +u_n(a-\mu u_n)-u(a-\mu u)  $. One can verify that $f_n$ satisfies the assumptions of Theorem \ref{linear-local-existence-thm}. Plus, the boundary condition
\begin{equation}
    \frac{\partial (U_n-U)}{\partial \nu } = |u_n|^{1+\gamma}-|u|^{1+\gamma}, \qquad \frac{\partial (V_n-V)}{\partial \nu }=0, \qquad x\in \partial \Omega,\, t \in (0, T_{\rm max}).
\end{equation}
We claim that $V_n \to V$ in $C^{2+\gamma}(\bar{\Omega_T} )$ for $\tau \geq 0$. Indeed, when $\tau >0$, we make use of \eqref{linear-est-local} and when $\tau =0$, we apply Schauder type estimate for elliptic equation to obtain that  $V_n \to V$ in $C^{2+\gamma}(\bar{\Omega_T} )$.  This leads to $f_n \to 0$ in $C^{\gamma}(\bar{\Omega}_T)$, combine with inequality \eqref{linear-est-local} entail that $U_n \to U$ in $C^{2+\gamma}(\bar{\Omega}_T)$. In order to apply the Leray-Schauder fixed point theorem we just have to prove that if $T$ is sufficiently small, and $R \geq 2(1+d(\Omega)^{1-\gamma}) \left \| u_0 \right \|_{C^{2+\gamma}(\bar{\Omega})}$, then $A(B_T(R)) \subset B_T(R)$. Indeed, 
\[
|Au(x,t)|\leq |Au(x,0)|+t\left\| D_t Au   \right\|_{C^{0}(\bar{\Omega})} \leq \left \| u_0 \right \|_{C^{0}(\bar{\Omega})}+TR'
\]
\[
\frac{|Au(x,t)-Au(x,s)|}{|t-s|^{\frac{1+\gamma}{2}}} \leq \left \| D_tAu \right \|_{C^0(\bar{\Omega}_T)}|t-s|^{\frac{1-\gamma}{2}}\leq R'T^{\frac{1-\gamma}{2}},
\]
and, 
\begin{align*}
    \frac{|D_xAu(x,t)-D_xAu(y,s)|}{|x-y|^\gamma+|t-s|^{\frac{\gamma}{2}}} &\leq |t-s|^{1-\frac{\gamma}{2}}\left \| D^2_xAu \right \|_{C^0(\bar{\Omega}_T)}+|x-y|^{1-\gamma}(s^{\frac{\gamma}{2}}+ \left \| D^2u_0 \right \|_{C^0(\bar{\Omega})} )\\
    &\leq T^{1-\frac{\gamma}{2}}R'+d(\Omega)^{1-\gamma} T^{\frac{\gamma}{2}}R'+d(\Omega)^{1-\gamma} \left \| D^2u_0 \right \|_{C^0(\bar{\Omega})}.
\end{align*}
These above estimates imply that
\[
\left \| Au \right \|_{C^{1+\gamma}(\bar{\Omega}_T)} \leq \frac{R}{2}+ R'T+R'T^{\frac{1-\gamma}{2}}+T^{1-\frac{\gamma}{2}}R'+d(\Omega)^{1-\gamma} T^{\frac{\gamma}{2}}R'.
\]
Since $R'$ is independent of $T$ for all $T<1$, we can choose $T$ sufficiently small as to have
\[
R'T+R'T^{\frac{1-\gamma}{2}}+T^{1-\frac{\gamma}{2}}R'+d(\Omega)^{1-\gamma} T^{\frac{\gamma}{2}}R' \leq \frac{R}{2}.
\]
This further implies that
\[
\left \| Au \right \|_{C^{1+\gamma}(\bar{\Omega}_T)} \leq R \text{ for all } u \in B_T(R).
\]
Thus $A$ has a fixed point in $B_T(R)$. Now if $u$ is a fixed point of $A$, $u \in C^{2+\gamma}(\bar{\Omega}_T) $ and it is a solution of \eqref{1.1}.\\
Secondly, the nonnegativity of solutions will be proved by the truncation method: Letting 
\[
\phi:= \min \left \{ u,0 \right \}
\]
and $\psi(t):= \frac{1}{2}\int_{\Omega} \phi^2 \, dx$, we see that $\psi$ is continuously differentiable with the derivative
\begin{align} \label{positive-1}
    \psi'(t)&= - \int_\Omega |\nabla \phi |^2 +a\int_\Omega \phi^2 +\int_{\partial \Omega} \phi |u|^{1+\gamma} \, dS +\chi \int_\Omega \phi \nabla \phi \cdot \nabla v  - \mu \int_{\Omega } \phi^3 \notag \\
    &\leq - \int_\Omega |\nabla \phi |^2 +a\int_\Omega \phi^2  +\chi \int_\Omega \phi \nabla \phi \cdot \nabla v  - \mu \int_{\Omega } \phi^3.
\end{align}
We make use of Young's inequality combined with the global boundedness of $|\nabla v|$ in $\bar{\Omega}_T$ to obtain
\begin{align}\label{positive-2}
    \chi \int_\Omega \phi \nabla \phi \cdot \nabla v  \leq \epsilon \int_\Omega |\nabla \phi|^2 +C\int_\Omega \phi^2,
\end{align}
for some $C>0$. We also have $- \mu \int_{\Omega } \phi^3 \leq C \int_{\Omega } \phi^2$, where $C=\mu \sup_{\Omega_T} |u(x,t)| $. This together with \eqref{positive-1}, \eqref{positive-2} implies that $\psi'(t) \leq  C \psi(t) $ for all $0<t<T$. By Gronwall's inequality and the initial condition $\psi(0)=0$, we imply that $\psi \equiv 0$ or $u \geq 0$. \\
Thirdly, we will prove that if $u_0 \not\equiv  0$ then $u$ is strictly positive in $\bar{\Omega_T}$ by a contradiction proof. Suppose that there exists $(x_0,t_0) \in \bar{\Omega}_T$ such that $\min_{\bar{\Omega}_T} u(x,t) = u(x_0,t_0) =0 $. By the strong parabolic maximum principle, we obtain $(x_0,t_0) \in \partial \Omega \times (0,T)$. However, it is a contradiction due to Hopf's lemma: 
 \[
 0 > \frac{ \partial u}{\partial \nu }(x_0,t_0) = |u(x_0,t_0)|^{1+\gamma} =0.
 \]
 Thus, $u>0$ and by similar arguments we also have $v>0$.\\
Finally, the uniqueness of classical solutions will be proved by a contradiction proof. Assuming $(u_1,v_1$ and $(u_2,v_2)$ are two positive classical solutions of the system \eqref{1.1}. Let $U:= u_1-u_2$, $V:= v_1-v_2$, then $(U,V)$ is a solution of the following system:
 \begin{equation} \label{local-existence-equation-2}
    \begin{cases} 
U_{t}=  \Delta U+ F,\, \qquad &x\in {\Omega}, \, t \in (0,T_{\rm max }), \\ 
 \tau V_t  = \Delta V+\gamma U - \beta V \qquad &x\in {\Omega},\, t \in (0,T_{\rm max }),
\end{cases} 
\end{equation}
where $F:= -\chi \nabla (u_1\nabla v_1-u_2\nabla v_2)+f(u_1)-f(u_2)$, and the boundary condition
\begin{equation}
    \frac{\partial U}{\partial \nu } = |u_1|^{1+\gamma}-|u_2|^{1+\gamma}, \qquad \frac{\partial V}{\partial \nu } = 0, \qquad x\in \partial \Omega,\, t \in (0, T_{\rm max}).
\end{equation}
By mean value theorem, there exists $z(x,t)$ between $u_1(x,t)$ and $u_2(x,t)$ such that \[
u_1(x,t)-u_2(x,t) =(u_1(x,t)-u_2(x,t))f'(z(x,t)).\] Multiplying the first equation of \eqref{local-existence-equation-2} by $U$ implies
\begin{align}\label{lon-1}
   \frac{1}{2} \frac{d}{dt}\int_\Omega U^2 \, dx &= -\int_\Omega |\nabla U|^2+ \int_{\partial \Omega} U(|u_1|^{1+\gamma}-|u_2|^{1+\gamma})\, dS \notag \\
    &+ \chi \int_\Omega (u_1\nabla v_1-u_2\nabla v_2)\cdot \nabla U + \int_\Omega U^2 f'(z).
\end{align}
We make use of the global boundedness property of $u_1,u_2$ in $\bar{\Omega}_T$, thereafter apply Sobolev's trace theorem, and finally Young's inequality to have 
\begin{align}\label{lon-2}
    \int_{ \partial \Omega } U(|u_1|^{1+\gamma}-|u_2|^{1+\gamma})\, dS \leq C\int_{\partial \Omega} U^2\, dS \leq \epsilon \int_{\Omega} |\nabla U|^2 +C(\epsilon)\int_\Omega U^2.
\end{align}
Since, 
\begin{align*}
    u_1\nabla v_1-u_2\nabla v_2 =U\nabla v_1 +u_2 \nabla V,
\end{align*}
we have 
\begin{align}\label{lon-3}
    \chi \int_\Omega \nabla U \cdot (u_1\nabla v_1-u_2\nabla v_2) &\leq C\int_\Omega U|\nabla U|+|\nabla U||\nabla V| \notag \\
    &\leq \epsilon \int_\Omega |\nabla U|^2 +C\int_\Omega U^2 +C\int_{\Omega} |\nabla V|^2.
\end{align}
We also have $\int_\Omega U^2 f'(z) \leq C\int_\Omega U^2$ where $C=\sup_{\min{ \{u_1,u_2\}} \leq z \leq \max{ \{u_1,u_2\}}} |f'(z)|$. Multiplying the second equation of \eqref{local-existence-equation-2} by $V$, and applying Young's inequality,  we obtain
\begin{align}\label{lon-5}
    \frac{d}{dt} \int_\Omega V^2 +\int_\Omega |\nabla V|^2 \leq C\int_\Omega U^2.
\end{align}
From \eqref{lon-1} to \eqref{lon-5}, we obtain
\begin{align}
     \frac{d}{dt}\left \{ \int_\Omega U^2+\int_\Omega V^2 \right \} \leq C\left \{ \int_\Omega U^2+\int_\Omega V^2 \right \}.
\end{align}
The initial conditions and Gronwall's inequality imply that $U\equiv V \equiv 0$, and thus there is a unique solution to the system \eqref{1.1}.
\end{proof}

\section{Preliminaries}\label{preliminaries}
We collect some useful tools that will frequently be used in the sequel. Let us begin with an extended version of the Gagliardo-Nirenberg interpolation inequality, which was established in \cite{Li+Lankeit}. 
\begin{lemma} \label{GN}
Let $\Omega$ be a  bounded and smooth domain of $\mathbb{R}^n$ with $n \geq 1$. Let $r\geq 1$, $0< q\leq p < \infty$, $s\geq 1$. Then there exists a constant $C_{GN}>0$ such that 
\begin{equation*}
    \left \| f \right \|^p_{L^p(\Omega)}\leq C_{GN}\left ( \left \| \nabla f \right \|_{L^r(\Omega)}^{pa}\left \| f \right \|^{p(1-a)}_{L^q(\Omega)} +\left \| f \right \|^p_{L^s(\Omega)}
 \right )
\end{equation*}
for all $f \in L^q(\Omega)$ with $\nabla f \in (L^r(\Omega))^n$, and $a= \frac{\frac{1}{q}-\frac{1}{p}}{\frac{1}{q}+\frac{1}{n}-\frac{1}{r}} \in [0,1]$.
\end{lemma}
Consequently, the next lemma is derived as follows:
\begin{lemma} \label{GNY}
If $\Omega$ be a  bounded and smooth domain of $\mathbb{R}^n$ with $n \geq 1$, and $f \in W^{1,2}(\Omega)$ then there exists a positive constant $C$ depending only on $\Omega$ such that the following inequality
\begin{align} \label{GNY1}
    \int_\Omega f^2 \leq C\eta \int_\Omega |\nabla f|^2 + \frac{C}{\eta^{\frac{n}{2}} } \left (   \int_\Omega |f| \right )^2
\end{align}
holds for all $\eta \in (0,1)$.
\end{lemma}
\begin{proof}
The Lemma follows from Lemma \ref{GN} by choosing $p=r=2$ and $q=s=1$ and Young's inequality.
\end{proof}
The following lemma gives estimates on solutions of the parabolic equations. For more details, see Lemma $2.1$ in \cite{Freitag-2018}.
\begin{lemma} \label{Para-Reg}
Let $p\geq 1$ and $q \geq 1$ satisfy 
\begin{equation*}
    \begin{cases}
     q &< \frac{np}{n-p},  \qquad \text{when } p<n,\\
     q &< \infty, \qquad \text{when } p=n,\\
      q &= \infty, \qquad \text{when } p>n.\\
     \end{cases}
\end{equation*}
Assuming $g_0 \in W^{1,q}(\Omega)$ and $g$ is a classical solution to the following system
\begin{equation}
    \begin{cases}
     g_t = \Delta g  - \zeta_1 g +\zeta_2 f &\text{in } \Omega \times (0,T), \\ 
\frac{\partial g}{\partial \nu} =  0 & \text{on }\partial \Omega \times (0,T),\\ 
 g(\cdot,0)=g_0   & \text{in } \Omega
    \end{cases}
\end{equation}
for some $T\in (0,\infty]$. If $f \in L^\infty \left ( (0,T);L^p(\Omega) \right ) $, then $g  \in L^\infty \left ( (0,T);W^{1,q}(\Omega) \right )$.
\end{lemma}
The following lemma provides estimates on the boundary (see Lemma 5.3 in \cite{Hiroshi}):
\begin{lemma} \label{convex}
Assume that $\Omega$ is a convex bounded domain, and that $f\in C^2(\bar{\Omega})$ satisfies $\frac{\partial f}{\partial \nu}=0$ on $\partial \Omega$. Then
\begin{equation*}
    \frac{\partial|\nabla f|^2}{\partial \nu} \leq 0 \qquad \text{ on } \partial \Omega.
\end{equation*}
\end{lemma}
The next lemma giving an useful estimate will later be applied in Section \ref{a priori estimates}. Interested readers are referred to \cite{Tao+Winkler2,Tian2} for more details about the proof. 
\begin{lemma}
Let $\Omega \subset \mathbb{R}^2$ be a bounded domain with smooth boundary, and let $p >1$ and $1\leq r<p$. Then there exists $C>0$ such that for each $\eta >0$, one can pick $C(\eta )>0$ such that
\begin{equation} \label{unif-GN}
    \left \| u \right \|^p_{L^p(\Omega)} \leq \eta \left \|\nabla u \right \|^{p-r}_{L^2(\Omega)}   \left \| u \ln{|u|} \right \|^r_{L^r(\Omega)} +C   \left \| u \right \|^p_{L^r(\Omega)}+C(\eta)
\end{equation}
holds for all $u \in W^{1,2}(\Omega)$.
\end{lemma}
The following lemma providing estimates on the boundary will be useful in Section \ref{a priori estimates}.
\begin{lemma} \label{Boundary-est}
If $r\geq \frac{1}{2}$, $p\in (1,\frac{3}{2})$, and $g \in C^1(\Bar{\Omega}),$ then for every $\epsilon>0$, there exists a constant $C=C(\epsilon,\Omega, p,r)$ such that 
\begin{equation} \label{ine-bdr}
    \int_{\partial \Omega} |g|^{p+2r-1} \leq \epsilon \int_{\Omega} |g|^{2r+1} +\epsilon \int_{\Omega} |\nabla g^r|^2 +C.
\end{equation}
\end{lemma}
\begin{proof}
Let $\phi :=|g|^r $, we have $\phi^{2+\frac{p-1}{r}} \in W^{1,1}(\Omega)$. Trace theorem, $W^{1,1}(\Omega) \to L^1(\partial \Omega)$, yields
\begin{align}\label{Boundary-est-proof-1}
    \int_{\partial \Omega} \phi ^{2+\frac{p-1}{r}} &\leq c_1 \int_{ \Omega} \phi ^{2+\frac{p-1}{r}} +(2+\frac{p-1}{r})c_1 \int_{ \Omega} \phi ^{1+\frac{p-1}{r}}|\nabla \phi| \notag \\
    &\leq c_1 \int_{ \Omega} \phi ^{2+\frac{p-1}{r}} +3c_1 \int_{ \Omega} \phi ^{1+\frac{p-1}{r}}|\nabla \phi|   
\end{align}
where $c_1=c_1(n,\Omega)>0$. By Young's inequality, the following holds for all $\epsilon>0$
\begin{align} \label{Boundary-est-proof-2}
    3c_1 \int_{ \Omega} \phi^{1+\frac{p-1}{r}}|\nabla \phi| \leq \epsilon \int_{ \Omega}|\nabla \phi|^2 +\frac{c_1^2}{\epsilon}\int_\Omega \phi^{2+\frac{2(p-1)}{r}}.
\end{align}
 We apply Young's inequality again to obtain a further estimate
\begin{align} \label{Boundary-est-proof-3}
  c_1 \int_{ \Omega} \phi ^{2+\frac{p-1}{r}}+\frac{c_1^2}{\epsilon}\int_\Omega \phi^{2+\frac{2(p-1)}{r}} \leq \epsilon \int_{ \Omega} \phi^{2+\frac{1}{r}} +c_2.
\end{align}
where $c_2$ depending on $\epsilon,p,r,n, \Omega$.  We complete the proof of \eqref{ine-bdr} by collecting \eqref{Boundary-est-proof-1},\eqref{Boundary-est-proof-2} and \eqref{Boundary-est-proof-3} together.
\end{proof}
The following lemma is similar to the previous lemma, however,  it is necessary to trace the dependency of $\eta, n$, and $p$ due to the involved technique of iteration used in Section \ref{globalbddn}.
\begin{lemma} \label{Boundary-regularity}
If $p\in (1,\frac{3}{2})$, and $g \in C^1(\Bar{\Omega}),$ then for every $\eta \in (0,\frac{1}{2})$, there exists a constant $c=c(n,\Omega, p)>0$ such that 
\begin{equation} \label{B-r-1}
    \int_{\partial \Omega} |g|^{p+2r-1} \leq \eta  \int_{\Omega} |g|^{2r+1} +\eta \int_{\Omega} |\nabla g^r|^2 +c \eta ^{\frac{n+2}{2p-3}} \left ( \int_{\Omega} |g|^{r}  \right )^2.
\end{equation}
\end{lemma}
\begin{proof}
We use the same notations as in the proof of Lemma \ref{Boundary-est} and assume $\epsilon \in (0,1)$. By Young's inequality,
\[
ab \leq \epsilon a^s +\frac{s-1}{s}(s\epsilon)^\frac{1}{1-s}b^\frac{s}{s-1}< \epsilon a^s +(s\epsilon)^\frac{1}{1-s}b^\frac{s}{s-1} , 
\]
for any $a,b \geq 0$, $\epsilon\in (0, \frac{1}{2})$ and $s>1$, we obtain
\begin{align} \label{b-r-proof-1}
    c_1 \int_\Omega \phi^{2+\frac{p-1}{r}} \leq c_1 \delta_1\int_\Omega \phi^{2+\frac{1}{r}} +c_1\left ( \frac{p-1}{\delta_1} \right )^{\frac{p-1}{2-p}}\int_\Omega \phi^2,
\end{align}
and 
\begin{align}\label{b-r-proof-2}
    \frac{c_1^2}{\epsilon} \int_\Omega \phi^{2+\frac{2p-2}{r}} \leq  \frac{c_1^2}{\epsilon} \delta_2\int_\Omega \phi^{2+\frac{1}{r}} + \frac{c_1^2}{\epsilon}\left ( \frac{2p-2}{\delta_2} \right )^{\frac{2p-2}{3-2p}}\int_\Omega \phi^2.
\end{align}
Choosing $\delta_1 = \frac{\epsilon}{2c_1}$ and $\delta_2= \frac{\epsilon^2}{2c_1^2}$, and then combining with \eqref{Boundary-est-proof-1}, \eqref{Boundary-est-proof-2}, \eqref{b-r-proof-1}, and \eqref{b-r-proof-2}, we obtain 
\begin{align} \label{b-r-proof-3}
    \int_{\partial \Omega} \phi ^{2+\frac{p-1}{r}} \leq \epsilon \int_\Omega |\nabla \phi|^2 +\epsilon \int_\Omega \phi^{2+\frac{1}{r}}+ (c_3\epsilon^{\frac{2p-1}{2p-3}} +c_4\epsilon^{\frac{p-1}{p-2}}), 
 \int_\Omega \phi^2,
\end{align}
where $c_3,c_4 >0$ independent of $r,\epsilon$. Notice that 
\begin{align*}
    c_3\epsilon^{\frac{2p-1}{2p-3}} +c_4\epsilon^{\frac{p-1}{p-2}} = \epsilon^{\frac{2p-1}{2p-3}} (c_3+c_4 \epsilon^{\frac{1}{(2-p)(3-2p)}}) \leq c_5 \epsilon^{\frac{2p-1}{2p-3}}, 
\end{align*}
where $c_5:=c_3+c_4$. This, together with \eqref{b-r-proof-3} implies that
\begin{align}\label{b-r-proof-4}
     \int_{\partial \Omega} \phi ^{2+\frac{p-1}{r}} \leq \epsilon \int_\Omega |\nabla \phi|^2 +\epsilon \int_\Omega \phi^{2+\frac{1}{r}} +c_5 \epsilon^{\frac{2p-1}{2p-3}} \int_\Omega \phi^2. 
\end{align}
We apply Lemma \ref{GNY} with $\eta = \min \left \{ 1, c_5^{-1}\epsilon^{\frac{2}{3-2p}} \right \}$ to have
\begin{align}\label{b-r-proof-5}
    c_5\epsilon^{\frac{2p-1}{2p-3}}\int_\Omega \phi^2 \leq \epsilon \int_\Omega |\nabla \phi|^2 +c_6\epsilon^{\frac{n+2}{2p-3}}\left ( \int_\Omega \phi \right )^2,
\end{align}
where $c_6>0$ independent of $r,\epsilon$. Collect \eqref{b-r-proof-4}, \eqref{b-r-proof-5} and substitute $\eta =2\epsilon$ proves \eqref{B-r-1}.
\end{proof}

\begin{lemma} \label{Boundary-est-3}
If $p\in (1,\frac{7}{5})$, $n=3$, and $(u,v)$ are in $C^1(\Bar{\Omega} \times (0,T_{\rm max}))$ and 
\begin{align}
    \int_\Omega |\nabla v(\cdot ,t)|^2 \leq A
\end{align}
holds for all $t\in (0,T_{\rm max})$, then for every $\epsilon>0$, there exists a constant $C=C(\epsilon,\Omega, p, A)$ such that 
\begin{equation} \label{ine-bdr-1}
    \int_{\partial \Omega} u^p|\nabla v|^2 \leq \epsilon \int_{\Omega} u^3 + |\nabla u|^2 + u^2|\nabla v|^2+\left | \nabla |\nabla v|^2 \right |^2 +C.
\end{equation}
holds for all $t\in (0,T_{\rm max})$.
\end{lemma}
\begin{proof}
By trace theorem $W^{1,1}(\Omega) \to L^{1}(\partial \Omega)$,
\begin{align} \label{Boundary-est-1-proof-1}
    \int_{\partial \Omega} |u|^{p}|\nabla v|^2 \leq c_1\int_{ \Omega}  |u|^{p}|\nabla v|^2  +c_1\int_{ \Omega} |u|^p\left |\nabla |\nabla v|^2  \right | +c_1p\int_{ \Omega} |u|^{p-1}|\nabla u||\nabla v|^2
\end{align}
where $c_1=c_1(\Omega)>0$. Apply Young's inequality yields
\begin{align}\label{Boundary-est-1-proof-2}
    c_1\int_{ \Omega}  |u|^{p}|\nabla v|^2 &\leq \frac{\epsilon }{2} \int_{\Omega} u^2|\nabla v|^2 +\frac{2-p}{2}\left ( \frac{p}{\epsilon c_1} \right )^{\frac{2-p}{p}}\int_{\Omega}|\nabla v|^2  \notag \\
    &\leq \frac{\epsilon }{2} \int_{\Omega} u^2|\nabla v|^2 +c_2
\end{align}
where $c_2:=\frac{2-p}{2}A\left ( \frac{p}{\epsilon c_1} \right )^{\frac{2-p}{p}}$. Note that $2p<3$, we apply Young's inequality to obtain
\begin{align}\label{Boundary-est-1-proof-3}
    c_1\int_{ \Omega} |u|^p\left |\nabla |\nabla v|^2  \right | &\leq \frac{\epsilon }{2} \int_{ \Omega}\left |\nabla |\nabla v|^2  \right |^2 + \frac{1}{2\epsilon} \int_\Omega u^{2p} \notag \\
    &\leq \frac{\epsilon }{2} \int_{ \Omega}\left |\nabla |\nabla v|^2  \right |^2 +\epsilon \int_{ \Omega} u^3 +c_3
\end{align}
where $c_3=c_3(\epsilon,\Omega,p)$. By Young's inequality,
\begin{align}\label{Boundary-est-1-proof-4}
    c_1p\int_{ \Omega} |u|^{p-1}|\nabla u||\nabla v|^2 &\leq \frac{\epsilon }{2}\int_{ \Omega} u^2|\nabla v|^2 +c_4\int_{ \Omega}|\nabla u|^{\frac{2}{3-p}}|\nabla v|^2\notag \\
    &\leq \frac{\epsilon }{2}\int_{ \Omega} u^2|\nabla v|^2 +\epsilon \int_{ \Omega} |\nabla u|^2 +c_5\int_{ \Omega} |\nabla v|^{\frac{6-2p}{2-p}}
\end{align}
where $c_4, c_5$ are positive and dependent on $\epsilon, \Omega, p$. Here, we use the condition $1<p<\frac{7}{5}$ to obtain $\frac{3-p}{2-p} < \frac{8}{3}$. By Young's inequality,
\begin{align}\label{Boundary-est-1-proof-5}
    c_5\int_{ \Omega} |\nabla v|^{\frac{6-2p}{2-p}} \leq \eta \int_{ \Omega} |\nabla v|^{\frac{16}{3}} +c_6
\end{align}
where $c_6=c_6(\eta,\epsilon,p,\Omega)$. In light of Gagliardo-Nirenberg inequality,
\begin{align} \label{Boundary-est-1-proof-6}
    \left \| |\nabla v|^2 \right \|_{L^{\frac{8}{3}}(\Omega)} &\leq c_{GN} \left \| \nabla |\nabla v|^2 \right \|^{\frac{3}{4}}_{L^2(\Omega)}\left \| |\nabla v|^2 \right  \|^{\frac{1}{4}}_{L^1(\Omega)}+c_{GN}\left \| |\nabla v|^2 \right  \|_{L^1(\Omega)} \notag\\
    &\leq c_{GN}A^{\frac{1}{4}}\left \| \nabla |\nabla v|^2 \right \|^{\frac{3}{4}}_{L^2(\Omega)} +c_{GN}A.
\end{align}
Hence 
\begin{align}\label{Boundary-est-1-proof-7}
    \eta \int_{ \Omega} |\nabla v|^{\frac{16}{3}} &\leq \eta \left ( c_{GN}A^{\frac{1}{4}}\left \| \nabla |\nabla v|^2 \right \|^{\frac{3}{4}}_{L^2(\Omega)} +c_{GN}A \right )^{\frac{8}{3}}\notag \\ 
    &\leq 2^{5/3} (c_{GN}A^{\frac{1}{4}})^{\frac{8}{3}} \eta \int_\Omega \left | \nabla |\nabla v|^2 \right |^2 +2^{5/3} (c_{GN}A)^{\frac{8}{3}}\eta.
\end{align}
Choosing $\eta$ such that ${2^{5/3} (c_{GN}A^{\frac{1}{4}})^{\frac{8}{3}}} \eta = \epsilon$, and plugging into \eqref{Boundary-est-1-proof-7},  \eqref{Boundary-est-1-proof-5}, and \eqref{Boundary-est-1-proof-4} respectively, we obtain 
\begin{align}\label{Boundary-est-1-proof-8}
      c_1p\int_{ \Omega} |u|^{p-1}|\nabla u||\nabla v|^2 \leq \epsilon \int_\Omega \left | \nabla |\nabla v|^2 \right |^2 + \frac{\epsilon }{2}\int_{ \Omega} u^2|\nabla v|^2 +\epsilon \int_{ \Omega} |\nabla u|^2+ c_7
\end{align}
where $c_7= c_6+2^{5/3} (c_{GN}A)^{\frac{8}{3}}\eta$. We finally complete the proof of \eqref{ine-bdr-1} by substituting \eqref{Boundary-est-1-proof-2}, \eqref{Boundary-est-1-proof-3} and \eqref{Boundary-est-1-proof-8} into \eqref{Boundary-est-1-proof-1}.
\end{proof}

\section{A priori estimates} \label{a priori estimates}
Let us first give a priori estimate for the parabolic-elliptic system, 
\begin{lemma} \label{lpboundedness-p-e-e}
If $\mu >0$ and $p \in (1,\frac{3}{2})$, for all $r \in (1, \frac{\chi \alpha}{(\chi \alpha-\mu)_+})$ then there exists $c= c(r,\left \| u_0 \right \|_{L^{r}(\Omega)})>0$ such that
 \begin{equation} \label{l1-ine1}
     \left \| u(\cdot,t) \right \|_{L^{r}(\Omega)} \leq c, \qquad \forall t\in (0,T_{\rm max}).
 \end{equation}
\end{lemma}
\begin{proof}
Multiplying the first equation in the system \eqref{1.1} by $u^{2r-1}$ yields
\begin{align}
    \frac{1}{2r}\frac{d}{dt}\int_{\Omega} u^{2r}  &= \int_{\Omega} u^{2r-1}u_t \notag\\
    &=  \int_\Omega u^{2r-1} \left [ \Delta u -\chi \nabla (u \nabla v) +f(u) \right ] \notag \\
    &=-\frac{2r-1}{r^2}\int_{\Omega} |\nabla u^r|^2  -\chi\frac{2r-1}{2r}\int_{\Omega}u^{2r} \Delta v  +\int_{\Omega}f(u)u^{2r-1} +\int_{\partial \Omega} u^{2r+p-1}\, dS  \notag \\
    &=-\frac{2r-1}{r^2}\int_{\Omega} |\nabla u^r|^2  +\frac{2r-1}{2r}\int_{\Omega}u^{2r}( \chi \alpha u -\chi \beta v ) \notag \\&+\int_{\partial \Omega} u^{2r+p-1}\, dS +a\int_{\Omega} u^{2r}-\mu \int_{ \Omega} u^{2r+1}.
\end{align}
Since $v\geq 0$, we have
\begin{align} \label{l1.1}
    \frac{d}{dt} \int_{\Omega} u^{2r} &\leq -\frac{2(2r-1)}{r}\int_{\Omega} |\nabla u^r|^2  -\left [ 2r\mu -\chi \alpha(2r-1) \right ]\int_\Omega u^{2r+1} \notag \\ 
    & +2r\int_{\partial \Omega} u^{2r+p-1}\, dS+2ra\int_{\Omega} u^{2r}.
\end{align}
By Lemma \ref{Boundary-est}, we obtain
\begin{align} \label{l1.4}
    2r \int_{\partial \Omega} u^{2r+p-1}\, dS \leq 2r\epsilon \int_\Omega |\nabla u^r|^2 + 2r\epsilon \int_\Omega u^{2r+1}+c_2.
\end{align}
We make use of Young's inequality to obtain
\begin{align} \label{l1.5}
    (2ra+1)\int_{\Omega}u^{2r} \leq \epsilon \int_{ \Omega} u^{2r+1} +c_3.
\end{align}
Collecting \eqref{l1.1}, \eqref{l1.4} and \eqref{l1.5}, we have
\begin{align} \label{l1.6}
    \frac{d}{dt}\int_{\Omega} u^{2r} +\int_{\Omega} u^{2r}  &\leq   \left [ 2r\epsilon -\frac{2(2r-1)}{r}  \right ]  \int_{\Omega}|\nabla u^r|^2 \notag 
    \\&-\left [  2r\mu -\chi \alpha(2r-1) -2\epsilon \right ]\int_{\Omega}u^{2r+1}+c_4.
\end{align}
If $\frac{\chi \alpha}{(\chi \alpha-\mu)_+
}>2r>1$, then selecting $\epsilon = \min \left \{  \frac{2r-1}{r^2}, \frac{2r\mu-\chi\alpha(2r-1)}{2} \right \}  $ and plugging into \eqref{l1.6}, we deduce
\begin{align}
    \frac{d}{dt}\int_{\Omega} u^{2r}+\int_{\Omega} u^{2r}  \leq c_2
\end{align}
This yields \eqref{l1-ine1}, hence the proof is complete.
\end{proof}
In the parabolic-parabolic case $\tau =1$, the following lemma gives us a priori bounds for solution of \eqref{1.1} with initial data \eqref{1.1.2} and the boundary condition \eqref{boundary-data}.
\begin{lemma} \label{i-es}
If $1<p<\frac{3}{2}$, and $(u,v)$ is a classical solution to \eqref{1.1}  with initial data \eqref{1.1.2} and the boundary condition \eqref{boundary-data} without the convexity assumption of $\Omega$, and $n \geq 2$ then there exists a positive constant $C$ such that
\begin{equation} \label{initialestimate}
    \int_\Omega (u(\cdot,t)+1) \ln{(u(\cdot,t)+1)} + \int_\Omega |\nabla v(\cdot,t)|^2  \leq C
\end{equation}
for all $t \in (0, T_{\rm max})$.
\end{lemma}

\begin{proof}
Let denote $y(t):= \int_\Omega (u(\cdot,t)+1) \ln{(u(\cdot,t)+1)} + \int_\Omega |\nabla v(\cdot,t)|^2  $, we have
\begin{align} \label{i-es.1}
    y'(t)&= \int_\Omega \left [ \Delta u -\chi \nabla \cdot (u \nabla v) +f(u)  \right ] \left [ \ln{(u+1)+1} \right ] \notag \\
    &+2\int_\Omega \nabla v \cdot \nabla  \left ( \Delta v +\alpha u- \beta v \right ) \notag \\
    &:= I_1+I_2.
\end{align}
By integration by parts, $I_1$ can be rewritten as
\begin{align} \label{i-es.2}
    I_1 &= -\int_\Omega \frac{|\nabla u|^2}{u+1} +\chi \int_\Omega \frac{u}{u+1}\nabla u \cdot \nabla v  
    +a\int_\Omega u\left [ \ln{(u+1)+1} \right ] \notag \\ 
    &-\mu\int_\Omega u^2\left [ \ln{(u+1)+1} \right ]  +\int_{\partial \Omega}u^p\left [ \ln{(u+1)+1} \right ] \, dS 
\end{align}
By integration by parts, Cauchy-Schwarz inequality and elementary inequality $\ln (u+1) \leq u$, we have
\begin{align}\label{i-es.10}
  \chi \int_\Omega \frac{u}{u+1}\nabla u \cdot \nabla v &= \chi \int_\Omega \nabla \left ( u- \ln (u+1) \right ) \cdot \nabla v \notag \\
 & = - \chi \int_\Omega  \left ( u- \ln (u+1) \right )\Delta v 
  \leq \frac{1}{2} \int_\Omega (\Delta v)^2 + \chi^2 \int_\Omega u^2.
\end{align}
One can verify that there exists $c_1(\mu, a)>0$ satisfying
\begin{equation*}
   u \left [ \ln{(u+1)+1} \right ] \leq \frac{\mu}{4 a}u^2 \left [ \ln{(u+1)+1} \right ]+c_1,
\end{equation*}
hence,
\begin{equation} \label{i-es.3}
   a\int_\Omega u\left [ \ln{(u+1)+1} \right ] \leq \frac{\mu }{4}\int_\Omega u^2\left [ \ln{(u+1)+1} \right ]    +c_1.
\end{equation}
In light of Sobolev's trace theorem, $W^{1,1}(\Omega) \hookrightarrow L^1(\partial \Omega)$, there exists $c_2(\Omega)>0$ such that
\begin{align} \label{i-es.4}
    \int_{\partial \Omega}u^p\left [ \ln{(u+1)+1} \right ] \, dS &\leq c_2\int_{\Omega}u^p\left [ \ln{(u+1)+1} \right ]+pc_2\int_{\Omega}u^{p-1}|\nabla u|\left [ \ln{(u+1)+1} \right ]\notag \\
    &+c_2\int_{\Omega}\frac{u^p}{u+1}|\nabla u| \left [ \ln{(u+1)+1} \right ].
\end{align}
By Young's inequality, we have
\begin{align} \label{i-es.6}
      pc_2\int_{\Omega}u^{p-1}|\nabla u|\left [ \ln{(u+1)+1} \right ] \leq \frac{1}{4}\int_\Omega \frac{|\nabla u|^2}{u+1}+pc_2\int_\Omega u^{2p-2}(u+1) \left [ \ln{(u+1)+1} \right ]^2,
\end{align}
and 
\begin{align} \label{i-es.5}
    c_2\int_{\Omega}\frac{u^p}{u+1}|\nabla u| \left [ \ln{(u+1)+1} \right ]\leq \frac{1}{4}\int_\Omega \frac{|\nabla u|^2}{u+1}+c^2_2\int_\Omega \frac{u^{2p}}{u+1} \left [ \ln{(u+1)+1} \right ]^2.
\end{align}
By the similar argument as in \eqref{i-es.3}, there exists $c_3(p,\Omega,\mu)>0$ such that
\begin{align} \label{i-es.7}
     c_2\int_{\Omega}u^p\left [ \ln{(u+1)+1} \right ] &+pc_2\int_\Omega u^{2p-2}(u+1) \left [ \ln{(u+1)+1} \right ]^2 \notag \\
     &+c^2_2\int_\Omega \frac{u^{2p}}{u+1} \left [ \ln{(u+1)+1} \right ]^2 \leq  \frac{\mu }{4}\int_\Omega u^2\left [ \ln{(u+1)+1} \right ]  +c_3.
\end{align}
From \eqref{i-es.4} to \eqref{i-es.7}, we obtain
\begin{equation} \label{i-es.8}
    \int_{\partial \Omega}u^p\left [ \ln{(u+1)+1} \right ] \, dS \leq \frac{1}{2}\int_\Omega \frac{|\nabla u|^2}{u+1}+\frac{\mu }{4}\int_\Omega u^2\left [ \ln{(u+1)+1} \right ]+c_3.
\end{equation}
Now, we handle $I_2$ as follows:
\begin{align} \label{i-es.9}
    I_2 &= -2\int_\Omega (\Delta v)^2 -2\beta \int_\Omega |\nabla v|^2+2\alpha \int_\Omega \nabla u \cdot \nabla v. 
\end{align}
By integration by part and Young's inequality, we have 
\begin{align} \label{i-es.10*}
    2\alpha \int_\Omega \nabla u \cdot \nabla v\leq \frac{1}{2}\int_\Omega (\Delta v)^2 + 2\alpha^2 \int_\Omega u^2.
\end{align}
One can verify that there exists $c_4(\alpha,\beta,\chi,\Omega)>0$ such that 
\begin{equation} \label{i-es.12}
(\chi^2+2\alpha^2)    \int_\Omega u^2 +2 \beta \int_\Omega (u+1)\ln{ (u+1)} \leq \frac{\mu }{4}\int_\Omega u^2\left [ \ln{(u+1)+1} \right ]+c_4.
\end{equation}
Collecting \eqref{i-es.2}, \eqref{i-es.3}, \eqref{i-es.8} and  from \eqref{i-es.9} to \eqref{i-es.12}, we obtain
\begin{equation}
    y'(t)+2\beta y(t) \leq -\frac{1}{2}\int_\Omega \frac{|\nabla u|^2}{u+1} - \frac{\mu }{4}\int_\Omega u^2\left [ \ln{(u+1)+1} \right ]+c_5 \leq c_5, \qquad \forall t \in (0,T_{\rm max}),
\end{equation}
where  $c_5:=c_1+c_3+c_4$. This, together with the Gronwall's inequality, yields
\begin{equation*}
    y(t) \leq e^{-2\beta t}y(0)+\frac{c_5}{2\beta }(1-e^{-2\beta t}) \leq C
\end{equation*}
where $C:= \max \left \{ y(0),\frac{c_5}{2\beta} \right \}$, and the proof of \eqref{initialestimate} is complete.
\end{proof}

The following lemma gives an $L^{2}$-bound in two-dimensional space for the parabolic-parabolic system.
\begin{lemma} \label{L2est}
If $\tau=1$, $n=2$, $1<p<\frac{3}{2}$, and $(u,v)$ is a classical solution to \eqref{1.1} with initial data \eqref{1.1.2} and the boundary condition then there exists a positive constant $C$ such that
\begin{equation} \label{L2estimate}
    \int_\Omega u^2(\cdot,t) + \int_\Omega |\nabla v(\cdot,t)|^4 \leq C
\end{equation}
for all $t \in (0, T_{\rm max})$.
\end{lemma}
\begin{proof}
Let denote 
\begin{equation*}
    \phi(t):= \frac{1}{2} \int_\Omega u^2 +\frac{1}{4} \int_\Omega |\nabla v|^4,
\end{equation*}
we have
\begin{align} \label{L2ets1}
    \phi'(t)&= \int_\Omega u\left [ \Delta u -\chi \nabla \cdot (u \nabla v)  +f(u)  \right ] \notag \\
    &+\int_\Omega |\nabla v|^2 \nabla v \cdot \nabla  \left ( \Delta v +\alpha u- \beta v \right ) \notag \\
    &:= J_1+J_2.
\end{align}
By integration by parts, we obtain 
\begin{align} \label{L2ets2}
    J_1 &=- \int_\Omega |\nabla u|^2 +\chi  \int_\Omega u \nabla u \cdot \nabla v +a\int_\Omega u^2 -\mu\int_\Omega u^3+ \int_{\partial \Omega}u^{p+1}\, dS.
\end{align}
By Young's inequality, we have
\begin{align} \label{L2ets3}
    \chi  \int_\Omega u \nabla u \cdot \nabla v +a\int_\Omega u^2\leq \frac{1}{2}\int_\Omega |\nabla u|^2 +\chi^2 \int_\Omega u^2 |\nabla v|^2.
\end{align}
In light of Sobolev's trace theorem, $W^{1,1}(\Omega) \hookrightarrow L^1(\partial \Omega)$, there exists  $ c_1:= c_1 (\Omega)>0$ such that
\begin{align} \label{L2ets4}
   \int_{\partial \Omega}u^{p+1}\, dS \leq c_1 \int_\Omega u^{p+1}+c_1(p+1)\int_\Omega u^{p}|\nabla u|.
\end{align}
Since $1<p<\frac{3}{2}$, we apply Young's inequality to obtain
\begin{align} \label{L2ets5}
    \int_{\partial \Omega}u^{p+1}\, dS &\leq \frac{\mu}{4}\int_{\Omega} u^3 + \frac{1}{4}\int_\Omega |\nabla u|^2+ {c_1(p+1)^2}\int_{\Omega} u^{2p} +c_2 \notag\\
    &\leq \frac{1}{4}\int_\Omega |\nabla u|^2+ \frac{\mu}{2}\int_{\Omega} u^3 +c_3,
\end{align}
where $c_2,c_3>0$ depending only on $p,\mu ,\Omega$.
To deal with $J_2$, we make use of the following pointwise identity
\begin{align*}
    \nabla v \cdot \nabla \Delta v = \frac{1}{2} \Delta( |\nabla v|^2) - |D^2v|^2 
\end{align*}
to obtain
\begin{align} \label{L2ets6}
    J_2&= -\frac{1}{2}\int_\Omega |\nabla|\nabla v|^2|^2-\int_\Omega |\nabla v|^2 |D^2 v|^2 \notag\\
    &+\alpha\int_\Omega |\nabla v|^2\nabla v\cdot \nabla u \notag \\
    &- \beta \int_\Omega |\nabla v|^4 +\frac{1}{2}\int_{\partial \Omega} \frac{\partial |\nabla v|^2 }{\partial \nu}|\nabla v|^2.
\end{align}
Applying Lemma \ref{convex} and the pointwise inequality $(\Delta v)^2 \leq 2 |D^2 v|^2$ to \eqref{L2ets6}, we deduce
\begin{align} \label{L2ets7}
    J_2  &\leq  -\frac{1}{2}\int_\Omega |\nabla|\nabla v|^2|^2 -\beta \int_\Omega |\nabla v|^4 \notag \\ &-\frac{1}{2}\int_\Omega |\nabla v|^2 |\Delta v|^2 +\alpha\int_\Omega |\nabla v|^2\nabla v\cdot \nabla u.
\end{align}
By integral by parts and Young's inequality, we obtain
\begin{align} \label{L2ets8}
   \alpha\int_\Omega |\nabla v|^2\nabla v\cdot \nabla u &=-\alpha\int_\Omega u \nabla |\nabla v|^2 \cdot \nabla v  -\alpha\int_\Omega u|\nabla v|^2 \Delta v \notag \\
    &\leq \frac{1}{4}\int_\Omega|\nabla |\nabla v|^2|^2  +  \frac{1}{4}\int_\Omega |\nabla v|^2 |\Delta v|^2 +2\alpha^2 \int_\Omega u^2|\nabla v|^2
\end{align}
Collecting from \eqref{L2ets1} to \eqref{L2ets8} yields
\begin{align} \label{L2ets10}
    \phi'+4\beta  \phi \leq - \frac{1}{4} \int_\Omega |\nabla u|^2-\frac{1}{2}\int_\Omega |\nabla|\nabla v|^2|^2-\frac{\mu}{2} \int_\Omega u^3+c_4\int_\Omega u^2|\nabla v|^2+c_5 \int_\Omega u^2+c_3
\end{align}
where $c_4,c_5$ are positive constants depending on $\alpha,\beta, \chi ,a$. By Young's inequality
\begin{align} \label{L2ets11}
  c_4 \int_\Omega u^2|\nabla v|^2\leq   c_4\epsilon \int_\Omega |\nabla v|^6+ \frac{ c_4}{\sqrt{\epsilon}}\int_\Omega u^3
\end{align}
By Gagliardo-Nirenberg inequality for $n=2$ and \eqref{initialestimate}, there exists $c_{GN} >0$ such that 
\begin{align} \label{L2ets12}
    \int_\Omega |\nabla v|^6 &\leq c_{GN} \left (\int_\Omega |\nabla |\nabla v|^2|^2  \right )\left (\int_\Omega  |\nabla v|^2 \right )+ c_{GN}\left (\int_\Omega  |\nabla v|^2 \right )^3 \notag \\
    &\leq c_6 \left (\int_\Omega |\nabla |\nabla v|^2|^2  \right )+ c_7,
\end{align}
where $c_6,c_7$ are positive constants depending on $c_{GN}$ and $\sup_{t \in (0,T_{\rm max})} \int_\Omega |\nabla v|^2$.
We make use of \eqref{unif-GN} for $n=2$ and \eqref{initialestimate} to obtain
\begin{align} \label{L2ets14}
    \int_\Omega u^3 &\leq \epsilon \left (  \int_\Omega |\nabla u|^2  \right )\left ( \int_\Omega u|\ln u| \right )+C \left ( \int_\Omega u \right )^3 +c(\epsilon) \notag\\
  &\leq c_{8}\epsilon \int_\Omega |\nabla u|^2+c_{9},
\end{align}
where $c_8:= \sup_{t\in (0,T_{\rm max})} \int_\Omega u|\ln u|$ and $c_9>0$ depending on $\epsilon$ and $\sup_{t\in (0,T_{\rm max})} \int_\Omega u$. In light of Young's inequality
\begin{align}\label{L2ets15}
    c_5 \int_\Omega u^2 \leq \frac{\mu}{4} \int_\Omega u^3 +c_{10}.
\end{align}
where $c_{10}>0$ depending on $c_5,\mu,|\Omega|$. Combining from \eqref{L2ets10} to \eqref{L2ets15}, we have
\begin{align} \label{L2ets16}
    \phi'+4\beta \phi &\leq \left (c_{4}c_8\sqrt{\epsilon} -\frac{1}{4} \right ) \int_\Omega |\nabla u|^2+\left ( c_4c_6\epsilon -\frac{1}{4} \right ) \int_\Omega |\nabla |\nabla v|^2|^2 +c_{11},   
\end{align}
where $c_{11}>0$ depending on $\epsilon$. Choosing $\epsilon$ sufficiently small and substituting into \eqref{L2ets16}, we obtain  
\begin{equation}
      \phi'+4\beta \phi \leq c_{11}.
\end{equation}
This, together with Gronwall's inequality yields $\phi(t) \leq C:= \max \left \{ \phi(0),\frac{c_{11}}{4\beta} \right \}$ for all $t \in (0,T_{\rm max})$, and the proof of Lemma \ref{L2est} is complete.
\end{proof}
The next lemma is the key step in the proof of the parabolic-parabolic system in three-dimensional space. 
\begin{lemma} \label{L2est-3d}
Let $(u,v)$ be a classical solution to \eqref{1.1} in a convex bounded domain $\Omega$ with smooth boundary. If $\tau=1$, $n=3$, $1<p< \frac{7}{5}$ and $\mu$ is sufficiently large, then there exists a positive constant $C$ such that
\begin{equation} \label{L2estimate-3d}
    \int_\Omega u^2(\cdot,t) + \int_\Omega |\nabla v(\cdot,t)|^4  \leq C
\end{equation}
for all $t \in (0, T_{\rm max})$.
\end{lemma}
\begin{remark}
 When we look at the proof of Theorem \ref{2dthm} carefully, $n=2$ is utilized to estimate $\int_\Omega u^2|\nabla v|^2$. For $n=3$, in order to eliminate this term we borrow the idea as in \cite{Winkler-2011, Tian}  by introducing an extra term $\int_\Omega u|\nabla v|^2$, in the function $\phi$ for Theorem \ref{2dthm}. Note that the term $u_t|\nabla v|^2$ will introduce $-\mu u^2|\nabla v|^2$, and a sufficiently large parameter $\mu$ will help the estimates.
\end{remark}
\begin{proof}
Let call $\psi(t):= \int_\Omega u^2+\int_\Omega |\nabla v|^4+ \frac{1}{3}\int_\Omega u|\nabla v|^2   $, we have
\begin{align} \label{L2ets1-3d-1}
    \psi'(t)&= 2\int_\Omega u\left [ \Delta u -\chi \nabla \cdot (u \nabla v)  +f(u)  \right ] \notag \\
    &+4\int_\Omega |\nabla v|^2 \nabla v \cdot \nabla  \left ( \Delta v +\alpha u- \beta v \right )  \notag \\
    &+\frac{2}{3}\int_\Omega u\nabla v \cdot \nabla  \left ( \Delta v +\alpha u- \beta v \right ) \notag \\
      &+\frac{2}{3}\int_\Omega |\nabla v|^2\left [ \Delta u -\chi \nabla \cdot (u \nabla v) +f(u)  \right ] \notag \\
    &:= K_1+K_2+K_3+K_4.
\end{align}
By integration by parts, $K_1$ can be written as:
\begin{align}\label{L2ets1-3d-2}
    K_1&= -2\int_\Omega |\nabla u|^2+2 \int_{\partial \Omega} u^{p+1} \, dS+2\chi \int_\Omega u \nabla u \cdot \nabla v \notag\\
    & + 2a  \int_\Omega u^2-2\mu \int_\Omega u^3.
\end{align}
By Lemma \ref{Boundary-est}, we obtain
\begin{align} \label{L2ets1-3d-3}
   2 \int_{\partial \Omega} u^{p+1} \, dS \leq 2 \epsilon_1 \int_{\Omega} u^3 +2 \epsilon_1 \int_{\Omega} |\nabla u|^2 +2c_1.
\end{align}
By Young's inequality, we have
\begin{align}\label{L2ets1-3d-4}
    2\chi \int_\Omega u \nabla u \cdot \nabla v \leq  \chi \epsilon_1  \int_{\Omega}|\nabla u|^2 +\frac{\chi}{\epsilon_1}\int_{\Omega}u^2|\nabla v|^2.
\end{align}
We also have
\begin{align}\label{L2ets1-3d-6}
    2a  \int_\Omega u^2 \leq 2a\epsilon_1 \int_{\Omega}u^3 +c_2,
\end{align}
where $c_2>0$ depending on $\epsilon_1$. From \eqref{L2ets1-3d-2} to \eqref{L2ets1-3d-6}, we obtain
\begin{align}\label{L2ets1-3d-7}
    K_1 &\leq \left [ (2 +\chi)\epsilon_1 -2    \right ] 
    \int_{\Omega}|\nabla u|^2+2\left [ (a+1)\epsilon_1 -\mu \right ]\int_{\Omega} u^3 +\frac{\chi }{\epsilon_1} \int_{\Omega}u^2|\nabla v|^2 +c_3,
\end{align}
where $c_3=2+c_2$. We choose $\epsilon_1 = \frac{1}{2} \min \left \{ \frac{2}{2+\chi }, \frac{\mu}{a+1} \right \}$ and substitute into \eqref{L2ets1-3d-7} to obtain
\begin{align} \label{L2ets1-3d-8}
    K_1 &\leq -\int_{\Omega}|\nabla u|^2- \mu \int_{\Omega} u^3 +\frac{\chi }{\epsilon_1} \int_{\Omega}u^2|\nabla v|^2 +c_3.
\end{align}
To deal with $K_2$, we use similar estimates to \eqref{L2ets6} and \eqref{L2ets7} in estimating $J_2$ in Lemma \ref{L2est} to obtain 
\begin{align}\label{L2ets1-3d-9}
    K_2+ 4 \beta \int_{\Omega}|\nabla v|^4 &\leq 2  (\alpha \epsilon_2-1)\int_{\Omega}|\nabla |\nabla v|^2|^2+ \left ( \frac{2\alpha}{\epsilon_2}+3\alpha^2  \right )\int_{\Omega}u^2|\nabla v|^2.
\end{align}
By substituting $\epsilon_2= \frac{1}{2\alpha}$  into \eqref{L2ets1-3d-9}, we deduce
\begin{align}\label{L2ets1-3d-10}
     K_2+ 4 \beta \int_{\Omega}|\nabla v|^4 &\leq - \int_{\Omega}|\nabla |\nabla v|^2|^2 +7 \alpha^2\int_{\Omega}u^2|\nabla v|^2
\end{align}
To deal with $K_3$, we make use of the following identity
\begin{align*}
    \nabla v \cdot \nabla \Delta v = \frac{1}{2} \Delta( |\nabla v|^2) - |D^2v|^2. 
\end{align*}
Rewriting $K_3$ as
\begin{align}
    K_3 &= -\frac{1}{3}\int_{\Omega} \nabla u \cdot \nabla |\nabla v|^2  -\frac{2}{3}\int_{\Omega} u|D^2v|^2  \notag \\
    &+ \frac{2 \alpha}{3}\int_{\Omega} u \nabla v \cdot \nabla u -\frac{2\beta}{3}\int_{\Omega} u |\nabla v|^2+\frac{1}{3}\int_{\partial \Omega} u \frac{\partial |\nabla v|^2}{\partial \nu}.
\end{align}
We drop the last term due to Lemma \ref{convex}, neglect the second term and apply Cauchy-Schwartz inequality
\[
ab \leq \epsilon a^2 + \frac{1}{4\epsilon} b^2, 
\] to the third and the forth terms with a sufficiently small $\epsilon$ to obtain
\begin{align}\label{L2ets1-3d-11}
    K_3+\frac{2\beta}{3} \int_\Omega  u|\nabla v|^2 &\leq \frac{1}{3} \int_{\Omega}|\nabla |\nabla v|^2|^2
    + \frac{\alpha^2}{8}\int_\Omega u^3 + \frac{1}{3}\int_\Omega |\nabla u|^2.
\end{align}
By integration by parts, $K_4$ can be rewritten as:
\begin{align}\label{L2ets1-3d-12}
K_4&= -\frac{1}{3} \int_\Omega   \nabla  |\nabla v|^2\cdot \nabla u +\frac{1}{3} \int_{\partial \Omega} u^p  |\nabla v|^2 \notag \\
&+\frac{\chi}{3} \int_\Omega   \nabla  |\nabla v|^2 \cdot   u \nabla v  +\frac{a}{3}\int_\Omega u  |\nabla v|^2-\frac{\mu}{3}\int_\Omega u^2  |\nabla v|^2 .
\end{align}
By Cauchy-Schwarz inequality, we obtain 
\begin{align}\label{L2ets1-3d-13}
   -\frac{1}{3} \int_\Omega   \nabla  |\nabla v|^2\cdot \nabla u  \leq \frac{1}{6} \int_\Omega \left | \nabla |\nabla v|^2 \right |^2 +\frac{1}{3}\int_\Omega |\nabla u|^2.
\end{align}
In light of Lemma \ref{Boundary-est-3} and Lemma \ref{i-es}, the following inequality
\begin{align} \label{L2ets1-3d-14}
    \frac{1}{3} \int_{\partial \Omega} u^p  |\nabla v|^2  \leq \frac{1}{3} \int_\Omega |\nabla u|^2+ u^2|\nabla v|^2+  \left | \nabla |\nabla v|^2 \right |^2+u^3 +c_1
\end{align}
holds for all $t \in (0,T_{\rm max})$, with some positive constant $c_1$. By Young's inequality, we obtain
\begin{align}\label{L2ets1-3d-15}
    \frac{\chi}{3} \int_\Omega   \nabla  |\nabla v|^2 \cdot   u \nabla v \leq \frac{1}{6} \int_\Omega  \left | \nabla |\nabla v|^2 \right |^2+\frac{\chi^2}{3} \int_\Omega  u^2|\nabla v|^2,
\end{align}
and  
\begin{align}\label{L2ets1-3d-16}
    \frac{a}{3}\int_\Omega u  |\nabla v|^2 \leq \frac{1}{3}\int_\Omega  u^2|\nabla v|^2 +c_2,
\end{align}
with some positive constant $c_2$.
Combining from \eqref{L2ets1-3d-12} to \eqref{L2ets1-3d-16}, we obtain
\begin{align}\label{L2ets1-3d-17}
    K_4 &\leq \frac{2}{3}\int_\Omega |\nabla u|^2+ \left | \nabla |\nabla v|^2 \right |^2+\frac{\chi^2-\mu+2}{3}\int_\Omega  u^2|\nabla v|^2 + \frac{1}{3} \int_\Omega u^3 +c_3,
\end{align}
with some positive constant $c_3$. Collecting \eqref{L2ets1-3d-8}, \eqref{L2ets1-3d-10}, \eqref{L2ets1-3d-11}, and  \eqref{L2ets1-3d-17}, we have
\begin{align}\label{L2ets1-3d-19}
    \psi' +4\beta \int_\Omega |\nabla v|^4 +\frac{2\beta }{3}\int_\Omega u|\nabla v|^2 &\leq  \left ( 7\alpha^2+\frac{\chi}{\epsilon_1}+\frac{\chi^2-\mu+2}{3} \right ) \int_\Omega  u^2|\nabla v|^2 \notag \\
      &+ \left ( \frac{1}{3}-\mu \right )\int_\Omega u^3 + \frac{ \alpha^2}{8}\int_\Omega u^2 +  c_5.
\end{align}
This leads to
\begin{align}
    \psi' +2\beta \psi &\leq \left ( 7\alpha^2+\frac{\chi}{\epsilon_1}+\frac{\chi^2-\mu+2}{3} \right ) \int_\Omega  u^2|\nabla v|^2 \notag \\
      &+ \left ( \frac{1}{3}-\mu \right )\int_\Omega u^3 + \frac{16\beta+\alpha^2}{8}\int_\Omega u^2 +  c_5.
\end{align}
By Young's inequality, for every $\epsilon>0$, there exists a positive constant $c_6=c_6(\epsilon)$ such that:
\begin{align}
    \frac{16\beta+\alpha^2}{8}\int_\Omega u^2 \leq \epsilon \int_\Omega u^3 +c_6.
\end{align}
Therefore, we need to choose $\mu_0 $ sufficiently large such that 
\begin{align*}
    \begin{cases}
     7\alpha^2+\frac{\chi}{\epsilon_1}+\frac{\chi^2-\mu_0+2}{3}&\leq 0 \\ 
 \frac{1}{3}+\epsilon -\mu_0 &\leq 0 \\ 
 \mu_0 &\geq \frac{2(a+1)}{2+\chi}.
    \end{cases}
\end{align*}
Therefore, if $\mu > \mu_0$, where
\begin{align}\label{L2ets1-3d-20}
    \mu_0 := \max \left \{ \frac{1}{3}, \frac{2(a+1)}{2+\chi}, 3\left ( \frac{\chi}{2+\chi} +7\alpha^2 +\frac{\chi^2+2}{2} \right ) \right \}
\end{align}
and then \eqref{L2ets1-3d-19} yields
\begin{align*}
     \psi' +2 \beta \psi \leq c_6.
\end{align*}
Applying Gronwall's inequality, we see that
\begin{align}
    \psi(t) \leq \max \left \{ \psi(0), \frac{c_6}{2\beta} \right \}
\end{align}
and thereby conclude the proof.
\end{proof}
\begin{remark}
 $\mu_0$ defined as in \eqref{L2ets1-3d-19} is not sharp. We leave the open question to obtain an optimal formula $\mu_0$.
\end{remark}

\section{Global boundedness} \label{globalbddn} 
In this section, we show that if $u$ is uniformly bounded in time under $\left\| \cdot  \right\|_{L^{r_0}(\Omega)}$, then it is also uniformly bounded in time under $\left\| \cdot  \right\|_{L^\infty(\Omega)}$. A method due to \cite{Alikakos2, Alikakos1} based on Moser-type iterations is implemented to prove the following theorem.
\begin{theorem} \label{bdduthm}
Let $r_0 > \frac{n}{2}$ and $(u,v)$ be a classical solution of \eqref{1.1} on $\Omega \times(0,T_{\rm max})$ with maximal existence time $T_{\rm max } \in (0, \infty]$. If
\[
\sup_{t\in (0,T_{\rm max})}\left \| u(\cdot,t) \right \|_{L^{r_0}(\Omega)} < \infty,
\]
then 
\[
\sup_{t\in (0,T_{\rm max})}\left (\left \| u(\cdot,t) \right \|_{L^{\infty}(\Omega)} +\left \| v(\cdot,t) \right \|_{W^{1,\infty}(\Omega)}    \right ) <\infty.
\]
\end{theorem}
\begin{remark}
One can also follow the argument as in \cite{Quittner} to prove the above theorem.
\end{remark}
\begin{remark}
  The $L^{\frac{n}{2}+}$-criterion for homogeneous Neumann boundary conditions has been studied in \cite{NAYM} for general chemotaxis systems and in \cite{Tian3} for the fully parabolic chemotaxis system, both with and without a logistic source. However, Theorem \ref{bdduthm} not only addresses nonlinear Neumann boundary conditions, but also employs a different analysis approach compared to \cite{NAYM,Tian3}. Instead of utilizing the semigroup estimate, the analysis in Theorem \ref{bdduthm} relies on the $L^p$-regularity theory for parabolic equations.
\end{remark}
The following lemma helps us to establish the iteration process from $L^r$ to $L^\infty$ when $r>n$.
\begin{lemma}\label{bddu}
Let $(u,v)$ be a classical solution of \eqref{1.1} on $(0,T_{\rm max})$ and \[ U_r := \max \left \{ \| u_0 \|_{L^\infty(\Omega)}, \sup_{t \in (0, T_{\rm max})} \| u(\cdot, t )  \|_{L^r{\Omega)}}  \right \}.   \] 
If $\sup_{t \in (0, T_{\rm max})} \| u(\cdot, t )  \|_{L^{r}{\Omega)}} < \infty$ for some   $r>n$, then there exists constants $A,B>0$ independent of $r$ such that 
 \begin{align} \label{bddu-ine}
    U_{2r} \leq  (Ar^{B})^{\frac{1}{2r}}U_r.  
 \end{align}
\end{lemma} 
\begin{proof}
The main idea is first to establish an inequality as follow:
\begin{align} \label{ine-prop}
    \frac{d}{dt}\int_{\Omega} u^{2r} +\int_{\Omega} u^{2r} \leq Ar^B \left ( \int_{\Omega} u^{r} \right )^2
\end{align}
and thereafter apply Moser iteration technique.
Therefore, the dependency of all constants in terms of $r$ will be traced carefully. Multiplying the first equation in the system \eqref{1.1} by $u^{2r-1}$ we obtain
\begin{align} \label{bddu.1}
    \frac{1}{2r}\frac{d}{dt}\int_{\Omega} u^{2r} &= \int_{\Omega} u^{2r-1}u_t \notag\\
    &=  \int_\Omega u^{2r-1} \left [ \Delta u -\chi \nabla (u \nabla v)  +a u -\mu u^2 \right ] \notag \\
    &=-\frac{2r-1}{r^2}\int_{\Omega} |\nabla u^r|^2 \, dx +J,
\end{align}
where 
\begin{align} \label{bddu.2}
     J &:= \int_\Omega u^{2r-1} \left [ -\chi \nabla (u \nabla v)  +a u -\mu u^2 \right ]\notag\\
    &= \chi \frac{2r-1}{2r} \int_\Omega \nabla u^{2r} \cdot \nabla v +a \int_\Omega u^{2r} \notag \\
    &-\mu \int_\Omega u^{2r+1} +\int_{\partial \Omega} u^{p+2r-1}\, dS \notag \\
    &= \chi  \frac{2r-1}{r} \int_\Omega u^{r} \nabla u^{r}\cdot \nabla v + a \int_\Omega u^{2r} \notag \\
    &-\mu \int_\Omega u^{2r+1} +\int_{\partial \Omega} u^{p+2r-1}\, dS. 
\end{align}
By Lemma \ref{Boundary-regularity}, for any $\epsilon \in (0,1)$, we obtain
\begin{align} \label{b-e}
    \int_{\partial \Omega} u^{p+2r-1}\, dS \leq \epsilon \int_{\Omega} |\nabla u^{r}|^2 + \epsilon \int_{\Omega} u^{2r+1} +c_1\left ( \frac{1}{\epsilon} \right )^{a_1} \left ( \int_{\Omega} u^r \right )^2.
\end{align}
where $a_1=\frac{n+2}{3-2p}$. We apply Young's inequality to the first two terms of the right hand side of \eqref{bddu.2} and combine with \eqref{b-e} to have 
\begin{align} \label{bddu.3}
    J &\leq 3\epsilon \int_{\Omega} |\nabla u^r|^2 + \frac{(2r-1)^2}{4r^2 \epsilon}\chi^2 \int _{\Omega} u^{2r} |\nabla v|^2   \notag \\
    &+a\int_\Omega u^{2r}+(\epsilon-\mu) \int_\Omega u^{2r+1} +c_1\left ( \frac{1}{\epsilon} \right )^{a_1} \left ( \int_{\Omega} u^r \right )^2,
\end{align}
for all $\epsilon \in (0,1)$. Lemma \ref{Para-Reg} asserts that $v$ is in $L^\infty\left ( (0,T); W^{1,\infty}(\Omega) \right )$. Thus,
\begin{equation*}
    \sup_{0<t<T}  \left \| \nabla v \right \|^2_{L^\infty}  = c_2 <\infty,
\end{equation*}
It follows from \eqref{bddu.1} and \eqref{bddu.3} that
\begin{align}\label{bddu.5}
   \frac{d}{dt} \int_\Omega u^{2r}+\int_\Omega u^{2r} & \leq 2r\left ( -\frac{2r-1}{r^2} +3\epsilon \right ) \int_{\Omega} |\nabla u^{r}|^2 +2r(\epsilon-\mu ) \int_\Omega u^{2r+1} \notag \\  
  &+ \left [ \frac{(2r-1)^2}{2r \epsilon}\chi^2c_2 + 2ra+1 \right ] \int_{\Omega} u^{2r}+ 2rc_1\left ( \frac{1}{\epsilon} \right )^{a_1} \left ( \int_{\Omega} u^r \right )^2 .
\end{align}
Substituting $\epsilon=\min \left \{ \frac{r-1}{3r^2}, \mu  \right \} $ into \eqref{bddu.5} and noticing that $\frac{1}{\epsilon} < 6r+\mu^{-1} $, we obtain
\begin{align} \label{*}
    \frac{d}{dt}\int_\Omega u^{2r} + \int_\Omega u^{2r} \leq -2  \int_{\Omega} |\nabla u^{r}|^2+c_3r^2 \int_\Omega u^{2r}+2rc_1\left ( 6r+\mu^{-1} \right )^{a_1} \left ( \int_{\Omega} u^r \right )^2 ,
\end{align}
where  $c_3$ are independent of $r$. We apply Lemma \ref{GNY}, and substitute into \eqref{*} to obtain the following inequality for all $\eta \in (0,1)$
\begin{align}
    \frac{d}{dt}\int_\Omega u^{2r} + \int_\Omega u^{2r} \leq (c_3r^2\eta -2)\int_{\Omega} |\nabla u^{r}|^2+\left ( \frac{c_4r^2}{\eta^{\frac{n}{2}}}+ 2rc_1\left ( 6r+\mu^{-1} \right )^{a_1} \right ) \left ( \int_{\Omega} u^r \right )^2,
\end{align}
 where $c_4>0$ independent of $r,\eta$. By substituting $\eta = \min \left \{  \frac{1}{c_3r^2}, 1 \right \}$ into this, we obtain the following
\begin{align} \label{**}
    \frac{d}{dt}\int_\Omega u^{2r} + \int_\Omega u^{2r} \leq \left [ c_5r^{n+2}+2rc_1( 6r+\mu^{-1})^{a_1} \right ]\left ( \int_{\Omega} u^r \right )^2,
\end{align}
where $c_5$ independent of $r$. Since $r>n$, there exists positive constants $\Bar{A},B$ independent of $r$ such that  
\[
c_5r^{n+2}+2rc_1( 6r+\mu^{-1})^{a_1} \leq \bar{A} r^B.
\]
Indeed, since $a_1>n+2$, we have
\[
c_5r^{n+2}+2rc_1( 6r+\mu^{-1})^{a_1} \leq \left [ \frac{c_5}{n^{a_1-n-1}} +2c_1 \left (6 +\frac{1}{\mu n} \right )^{a_1} \right ]r^{a_1+1}.
\]
By substituting this into \eqref{**}, we obtain \eqref{ine-prop}. Applying Gronwall's inequality, we have 
\[
\int_\Omega u^{2r}(\cdot,t)  \leq \max \left \{ \bar{A} r^B U_r^{2r} ,\int_\Omega u_0^{2r}    \right \}
\]
This entails
\begin{align*}
    \left \| u(\cdot, t)  \right \|_{L^{2r}(\Omega)} \leq \max \left \{ (\bar{A}r^B)^\frac{1}{2r} U_r,  |\Omega|^{\frac{1}{2r}}  \left \| u_0 \right \|_{L^{\infty}(\Omega)}  \right \},
\end{align*}
and further implies that 
\begin{align*}
    U_{2r} \leq ({A} r^B)^\frac{1}{2r} U_r
\end{align*}
where ${A} = \max \left \{ \bar{A}, |\Omega|n^{-B} \right \}$. The proof of \eqref{bddu-ine} is complete.
\end{proof}
The next lemma will fill in the gap of the iteration process from $L^r$ to $L^n$ when $\frac{n}{2}<r \leq n$.
\begin{lemma} \label{r-2r}
Let $(u,v)$ be the classical solution to \eqref{1.1} on $\Omega \times(0,T_{\rm max})$ with maximal existence time $T_{\rm max } \in (0, \infty]$. If $u \in L^\infty \left ( (0,T_{\rm max}); L^r(\Omega)  \right )$ for some $ \frac{n}{2}<r \leq n$, then  $u \in L^\infty \left ( (0,T_{\rm max}); L^{2r}(\Omega)  \right )$.
\end{lemma}
\begin{proof}
 By Lemma \ref{Para-Reg} we see that $v$ is in $L^\infty\left ( (0,T); W^{1,q}(\Omega) \right )$ for $q<\frac{rn}{n-r}$ if $r<n$ and any $q<\infty$ if $r=n$. Set
\begin{align}
    \lambda := \left\{\begin{matrix}
\frac{2n}{n-2} & \text{ if }n \geq 3 \\ 
 \frac{1}{r-1}& \text{ if } n=2,
\end{matrix}\right.
\end{align}
 and apply Holder's inequality yields
\begin{equation}
    \int_{\Omega}u^{2r}|\nabla v|^2 \leq \left ( \int_{\Omega}u^{2r+\lambda} \right )^{\frac{2r}{2r+\lambda}}\left (\int_{\Omega} |\nabla v|^{\frac{2(2r+\lambda)}{\lambda}} \right )^{\frac{\lambda}{2r+\lambda}}.
\end{equation} 
Since $\frac{n}{2}<r<n$, we find that
\[
\frac{2(2r+\lambda)}{\lambda} < \frac{rn}{n-r} ,
\]
therefore $v$ belongs to $L^\infty\left ( (0,T); W^{1,\frac{2(2r+\lambda)}{\lambda}}(\Omega) \right )$. Clearly, $v$ is in $L^\infty\left ( (0,T); W^{1,\frac{2(2r+\lambda)}{\lambda}}(\Omega) \right )$ when $r=n$.  Thus,
\begin{equation*}
    \sup_{0<t<T} \left \| \nabla v \right \|^2_{L^{\frac{2(2r+\lambda)}{\lambda}}}  = c_5 <\infty,
\end{equation*}
notice that
\begin{equation} \label{bddu.13}
    \int_{\Omega} u^{2r} \leq |\Omega|^{\frac{\lambda}{2r+\lambda}} \left ( \int_{\Omega} u^{2r+\lambda} \right )^{\frac{2r}{2r+\lambda}} 
\end{equation}
By the similar argument as in \eqref{bddu.5} and combine with \eqref{bddu.13}, it follows that
\begin{align}\label{bddu.13'}
  \frac{d}{dt} \int_\Omega u^{2r} +\int_\Omega u^{2r} & \leq 2r\left ( -\frac{2r-1}{r^2} +3\epsilon \right ) \int_{\Omega} |\nabla u^{r}|^2 +2r(\epsilon-\mu) \int_\Omega u^{2r+1} \notag \\  
  &+ c_6 \left ( \int_{\Omega}u^{2r+\lambda} \right )^{\frac{2r}{2r+\lambda}}+ 2rc_1\left ( \frac{1}{\epsilon} \right )^{a_1} \left ( \int_{\Omega} u^r \right )^2 ,
\end{align}
where $c_6=\left [ \frac{(2r-1)^2}{2r \epsilon}\chi^2c_5 + 2ra|\Omega|^{\frac{\lambda}{2r+\lambda}}+1 \right ]$.
Substitute $\epsilon = \min \left \{ \frac{r-1}{3r^2}, \mu  \right \}$ into this yields
\begin{equation} \label{bd-1}
    \frac{d}{dt} \int_\Omega u^{2r} + \int_\Omega u^{2r}  \leq -2\int_{\Omega} |\nabla u^r|^2+ c_6\left ( \int_{\Omega}u^{2r+\lambda} \right )^{\frac{2r}{2r+\lambda}} +c_7,
\end{equation}
where $c_7:= 2rc_1U_r^{2r}$. Apply GN inequality, and then Young's inequality yields 
\begin{align} \label{bd-2}
    c_6\left ( \int_{\Omega}u^{2r+\lambda} \right )^{\frac{2r}{2r+\lambda}} &\leq c_6C_{GN} \left ( \int_\Omega |\nabla u^r|^2 \right )^s\left ( \int_\Omega u^r \right )^{2(1-s)} +c_6C_{GN}\left ( \int_\Omega u^r  \right )^2 \notag \\
    &\leq c_6C_{GN}U_r^{2r(1-s)}\left ( \int_\Omega |\nabla u^r|^2 \right )^s +c_6C_{GN}U_r^{2r} \notag \\
    &\leq  \int_\Omega |\nabla u^r|^2 +c_8,
\end{align}
where $s=\frac{2n(r+\lambda)}{(n+2)(2r+\lambda)} \in (0,1)$, and $c_8=c_8(r,n,\Omega,U_r)$. From \eqref{bd-1} and \eqref{bd-2}, we obtain
\begin{align}
     \frac{d}{dt} \int_\Omega u^{2r} + \int_\Omega u^{2r}  \leq c_9,
\end{align}
where $c_9=c_7+c_8$. Apply Gronwall's inequality to this yields
\[
\sup_{t \in (0,T_{\rm max})} \left \| u(\cdot,t) \right \|_{L^{2r}(\Omega)} \leq \max \left \{ \left \| u_0  \right \|_{L^\infty(\Omega)},c_9 \right \}.
\]
and thereby conclude the proof.
\end{proof}
We are now ready to prove Theorem \ref{bdduthm}.
\begin{proof}[Proof of Theorem \ref{bdduthm}]
If $r_0 \in \left ( \frac{n}{2}, n  \right ]$, then
\[
\sup_{t \in (0,T_{\rm max})} \left \| u(\cdot,t) \right \|_{L^{2r_0}(\Omega)} < \infty.
\]
 due to Lemma \ref{r-2r}. Since $2r_0>n$ meets the criteria of Lemma \ref{bddu}, the following inequality holds
\begin{align} \label{f}
    U_{2^{k+1}r_0} \leq \left (  A(2^{k}r_0)^B \right )^{\frac{1}{2^{k+1}r_0}}U_{2^kr_0}.
\end{align}
for all integers $k \geq 1$. If $r_0>n$, then \eqref{f} is an immediate consequence of Lemma \ref{bddu}.
We take $\log$ of the above inequality to obtain
\[
\ln  U_{2^{k+1}r_0} \leq a_k+ \ln  U_{2^{k}r_0},
\]
where
  \begin{align*}
      a_k&= \frac{\ln A}{2^{k+1}r_0}+ \frac{Bk \ln 2}{2^{k+1} r_0} + \frac{B \ln r_0}{2^{k+1}r_0}.
  \end{align*}
One can verify that 
\begin{align*}
    \sum_{k=1}^{\infty} a_k = \frac{\ln \left (A(4r_0)^B  \right ) }{2r_0}.
\end{align*}
Thus, we obtain
\begin{equation}
     U_{2^{k+1}r}\leq A^\frac{1}{2r_0}(4r_0)^{\frac{B}{2r_0}}U _{2r_0} 
  \end{equation}
  for all $k\geq 1$. Send $k \to \infty$ yields
  \begin{equation}
       U_{\infty}\leq  A^\frac{1}{2r_0}(4r_0)^{\frac{B}{2r_0}}U _{2r_0}. 
  \end{equation}
 This asserts that $u \in L^\infty \left ( (0,T_{\rm max}); L^\infty (\Omega) \right )$, and thereafter Lemma \ref{Para-Reg} yields that $v \in L^\infty \left ( (0,T_{\rm max}); W^{1,\infty}(\Omega) \right )$. 
\end{proof}

\section{Proof of main theorems}\label{proof of main theorems}
We are now ready to prove our main results. Throughout this section, unless specified otherwise, the notation $C$ represents constants that may vary from time to time. Let us begin with the proof of the parabolic-elliptic system.
\begin{proof}[Proof of Theorem \ref{p-e-e}]
In case $\mu >\frac{n-2}{n}\chi \alpha$, we obtain $\frac{\chi \alpha}{(\chi \alpha-\mu)_+}>\frac{n}{2}$. We first apply Lemma
\ref{lpboundedness-p-e-e} to have $u \in L^\infty\left ((0,T_{\rm max}); L^{q}(\Omega)  \right )$  for any $q\in \left ( \frac{n}{2}, \frac{\chi \alpha}{(\chi \alpha-\mu)_+} \right ) $ and thereby conclude that $u \in L^\infty\left ((0,T_{\rm max}); L^{\infty}(\Omega)  \right )$ thank to Theorem \ref{bdduthm}. \\
When $\mu =\frac{n-2}{n}\chi \alpha$, let $w= u^{\frac{n}{4}}$, and apply trace embedding Theorem $W^{1,1}(\Omega) \to L^1(\partial \Omega)$, we have
\begin{align} \label{proof.thm.1}
    \int_{\partial \Omega} w^{2+\frac{4(p-1)}{n}} &\leq C \int_\Omega  w^{2+\frac{4(p-1)}{n}} +C \left (2+\frac{4(p-1)}{n} \right )\int_\Omega w^{1+\frac{4(p-1)}{n}}|\nabla w| \notag \\
    &\leq  C \int_\Omega  w^{2+\frac{4(p-1)}{n}} +3C\int_\Omega w^{1+\frac{4(p-1)}{n}}|\nabla w| \notag \\
    &\leq C \int_\Omega  w^{2+\frac{4(p-1)}{n}} + \frac{\epsilon }{2}\int_\Omega |\nabla w|^2 +C(\epsilon) \int_\Omega w^{2+\frac{8(p-1)}{n}},
\end{align}
where the last inequality comes from Young's inequality for any arbitrary $\epsilon>0$. By Lemma \ref{GN}, we have 
\begin{align} \label{proof.thm.2}
     \int_\Omega w^{2+\frac{8(p-1)}{n}} \leq C \left ( \int_\Omega |\nabla w|^2 \right )^\frac{\bar{p}a}{2} \left ( \int_\Omega w^{\frac{4}{n}} \right )^{\frac{n\bar{p}(1-a)}{4}} + C \left ( \int_\Omega w^{\frac{4}{n}} \right )^{\frac{n\bar{p}}{4}},
\end{align}
where 
\begin{align*}
    \bar{p} &=2+\frac{8(p-1)}{n}, \\
    a &=\frac{n^2(n+4(p-1)+2)}{(n+4(p-1))(n^2-2n+4)},\\
    \frac{\bar{p}a}{2} &= \frac{n^2+4(p-1)n-2n}{n^2-2n+4} <1, \qquad \text{ since } p<1+\frac{1}{n}.
\end{align*}
We make use of uniformly boundedness of $\int_\Omega u$ and then apply Young's inequality into \eqref{proof.thm.2} to obtain:
\begin{align}
    \int_\Omega w^{2+\frac{8(p-1)}{n}} \leq \frac{\epsilon }{2} \int_\Omega |\nabla w|^2 +C(\epsilon),
\end{align}
for any $\epsilon>0$. We apply Young's inequality in \eqref{proof.thm.1} and then use \eqref{proof.thm.2} to have
\begin{align}
    \int_{\partial \Omega} w^{2+\frac{4(p-1)}{n}} \leq \epsilon \int_\Omega |\nabla w|^2 + C(\epsilon).
\end{align}
This implies that
\begin{align} \label{proof.thm.3}
    \int_{\partial \Omega} u^{\frac{n}{2}+p-1} \leq \epsilon \int_\Omega |\nabla u^{\frac{n}{4}}|^2 +C(\epsilon).
\end{align}
Substitute $r=\frac{n}{4}$ and $\mu =\frac{n-2}{n}\chi \alpha$ into \eqref{l1.1}, we have
\begin{align} \label{proof.thm.4}
    \frac{d}{dt} \int_\Omega u^{\frac{n}{2}} \leq -\frac{4(n-2)}{n} \int_{\Omega} |\nabla u^{\frac{n}{4}}|^2 +\frac{n}{2} \int_{\partial \Omega} u^{\frac{n}{2}+p-1}\, dS +\frac{na}{2} \int_\Omega u^{\frac{n}{2}}.
\end{align}
We apply Lemma \ref{GN} and uniform boundedness of $\int_\Omega u$ to obtain that
\begin{align}
   \int_\Omega u^{\frac{n}{2}} \leq \epsilon \int_\Omega |\nabla u^{\frac{n}{4}}|^2 +C(\epsilon).
\end{align}
This, together with \eqref{proof.thm.3} and \eqref{proof.thm.4} with $\epsilon$ sufficiently small implies that 
\[
\frac{d}{dt} \int_\Omega u^{\frac{n}{2}} + \int_\Omega u^{\frac{n}{2}}  \leq C.
\]
Thus, by Gronwall's inequality we obtain that 
\begin{align}  \label{proof.thm.5}
    \sup_{t \in (0,T_{\rm max})} \int_\Omega u^{\frac{n}{2}} < \infty.
\end{align}
For any $\epsilon\in (0,1)$, we choose $\frac{n}{4}<r < \frac{n}{4}+\frac{n\epsilon}{4\chi \alpha} $ and substitute into \eqref{l1.6} to obtain that
\begin{align}
    \frac{d}{dt} \int_\Omega u^{2r}+\int_\Omega u^{2r} \leq \left [ 2r\epsilon -\frac{2(2r-1)}{r} \right ] \int_\Omega |\nabla u^r|^2 +3\epsilon \int_\Omega u^{2r+1} +C.
\end{align}
We apply Lemma \ref{GN} and \eqref{proof.thm.5} to obtain
\begin{align} \label{proof.thm.6}
     \int_\Omega u^{2r+1} &\leq C \left ( \int_\Omega |\nabla u^r|^2 \right )\left ( \int_\Omega  u^{\frac{n}{2}} \right )^{\frac{2}{n}}+\left ( \int_\Omega  u \right )^{2r+1} \notag \\
     &\leq C\int_\Omega |\nabla u^r|^2 +C,
\end{align}
where $C $ is independent of $r$. Therefore, we have
\begin{align} \label{proof.thm.7}
     \frac{d}{dt} \int_\Omega u^{2r}+\int_\Omega u^{2r} \leq \left [ 2r\epsilon +3C\epsilon - \frac{2(2r-1)}{r} \right ] \int_\Omega |\nabla u^r|^2 +C.
\end{align}
We now have to choose $\epsilon$ and $r>\frac{n}{4}$ such that 
\begin{align}  \label{proof.thm.8}
    \frac{4\chi \alpha}{n} \left ( r- \frac{n}{4} \right )< \epsilon <\frac{2(2r-1)}{r(2r+C)},
\end{align}
which is possible for any $r$ satisfying 
\[
\frac{n}{4} <r <\frac{n}{4} \left ( \frac{4(n-2)}{\frac{n}{2} +\frac{n}{2\chi \alpha}+3C}+1\right ).
\]
This, together with \eqref{proof.thm.7}, \eqref{proof.thm.8} implies that there exists some $r_0>\frac{n}{2}$ such that $u \in L^\infty \left ( (0,T_{\rm max}); L^{r_0}(\Omega) \right )$. We finally complete the proof by applying Theorem \ref{bdduthm}
\end{proof}
Next we prove the main theorems for the parabolic-parabolic system in two- and three-dimensional space.
\begin{proof}[Proof of Theorem \ref{2dthm} and Theorem \ref{3d}]
Theorem \ref{2dthm} and Theorem \ref{3d} are immediate consequences of Lemma \ref{L2est}, Lemma \ref{L2est-3d} and Theorem \ref{bdduthm}.
\end{proof}

\section*{Acknowledgement}
The author would like to express deepest appreciation to Professor Zhengfang Zhou for his unwavering encouragement, support, and numerous fruitful discussions. The author is profoundly grateful to Professor Michael Winkler for his valuable comments and suggestions. Lastly, the author would like to express gratitude to the referee for carefully reading the manuscript and providing constructive feedback as well as helpful references.

\printbibliography

\end{document}